\documentclass[reqno,12pt]{amsart} 
\usepackage{vmargin}
\setpapersize{A4}
\usepackage{amsmath} 

\usepackage{amsfonts}   
\usepackage{amssymb}      
\usepackage{eufrak}      
\usepackage{amscd}      
\usepackage{amsthm}      
\usepackage{epsfig}      
\usepackage{amstext}      
\usepackage[all,line,dvips]{xy} 
\CompileMatrices %goer at xy-matricerne koerer hurtigere.

\newcommand{\Ad}{\operatorname{Ad}}

\newcommand{\id}{\operatorname{id}} 
\newcommand{\Aut}{\operatorname{Aut}}

\newcommand{\Span}{\operatorname{Span}}

\newcommand{\Per}{\operatorname{Per}}

 \newcommand{\supp}{\operatorname{supp}}

\newcommand{\Is}{\operatorname{Is}}

\newcommand{\ev}{\operatorname{ev}}

\newcommand{\val}{\operatorname{val}}

\newcommand{\BD}{\operatorname{BD}}

\newcommand{\J}{C^*_r\left({J_R}\right)}
\newcommand{\F}{C^*_r\left({F_R}\right)}
\newcommand{\Real}{\operatorname{Re}}
\newcommand\RDA{\textup{\fontsize{4pt}{4pt}\selectfont RDA}}
\newcommand{\C}{\overline{\mathbb C}}
\newcommand{\Ro}{\operatorname{RO}}

\newcommand{\crt}{\operatorname{Crit}}

\DeclareMathOperator{\orb}{Orb}

   \theoremstyle{plain}%default
   \newtheorem{thm}{Theorem}[section]
   \newtheorem{prop}[thm]{Proposition}
   \newtheorem{lemma}[thm]{Lemma}  
   \newtheorem{cor}[thm]{Corollary}
   \theoremstyle{definition}
   
   \newtheorem{defn}[thm]{Definition}
   \newtheorem{example}[thm]{Example}
   \theoremstyle{remark}

\usepackage{graphicx}

   \numberwithin{equation}{section}

        \date{\today}
\title[The $C^*$-algebra of a
rational map]{The groupoid $C^*$-algebra of a rational map}
  
\author{Klaus Thomsen}

\date{\today}

\email{matkt@imf.au.dk}
\address{Institut for matematiske fag, Ny Munkegade, 8000 Aarhus C, Denmark}

\begin{document}

\maketitle

\begin{abstract} This paper contains a quite detailed description of the
  $C^*$-algebra arising from the transformation groupoid of a
   rational map of
  degree at least two on the Riemann sphere. The algebra is decomposed
  stepwise via extensions of familiar $C^*$-algebras whose nature
  depend on the structure of the Julia set and the stable regions in the
  Fatou set, as well as on the behaviour of the critical points.
 \end{abstract}

\section{Introduction and presentation of results}

In non-commutative geometry it is a basic principle, referred to as
Connes' dictum in \cite{Kh}, that a quotient space should
be replaced with a non-commutative algebra, preferably a
$C^*$-algebra, in the cases where the topology of the quotient is ill-behaved. Following this dictum the procedure should go via an
intermediate step which first produces a groupoid, and the
non-commutative algebra should then arise as the convolution algebra
of the groupoid. The structure of the resulting non-commutative algebra offers to compensate for the
poor topology which the quotient very often is equipped with, and at
the same time it encodes the equivalence relation defining the
quotient space which is otherwise lost.

A standard example of the construction is
the classical crossed product arising from a group acting on a locally
compact Hausdorff space which in this picture is a non-commutative
substitute for the quotient of the space under the orbit equivalence
relation given by the action. There are several other examples of this type
of construction arising from dynamical systems. The $C^*$-algebras
introduced by D. Ruelle in \cite{Ru} are non-commutative algebras
representing the quotient space under the homoclinic equivalence
relation arising from a hyperbolic homeomorphism, while the extension
of Ruelles approach by I. Putnam (\cite{Pu}) also allows to consider the quotient by the heteroclinic
equivalence relations of the same type of dynamics. The full orbit equivalence
relation arising from a non-invertible continuous self map can also
serve as input when the map is locally injective. For local
homeomorphisms the construction was developed in stages by J. Renault
(\cite{Re}), V. Deaconu (\cite{De}) and C. Anantharaman-Delaroche (\cite{An}), while the extension to locally injective maps was carried
out in \cite{Th1}. In all cases, including the work of Ruelle and
Putnam, a major problem is to equip the natural groupoid with a
sufficiently nice topology which allows the construction of the
convolution $C^*$-algebra. The best one can hope for is to turn the
groupoid into a locally compact Hausdorff groupoid in such a way that the
range map becomes a local homeomorphism. In this case the groupoid is said
to be \'etale. These crucial properties come relatively cheap for the
transformation groupoid of Renault, Deaconu and
Anantharaman-Delaroche, while it is harder to obtain them for the
groupoids in Putnam's construction (\cite{PS}). In a recent
work (\cite{Th2}) it was shown that it is possible to formulate the
definition of the transformation groupoid of a homeomorphism or a
local homeomorphism in such a way that it not only makes sense for a
larger class of continuous maps, but also retains the structure of a locally compact second countable
Hausdorff groupoid with the important \'etale property. This class of
continuous maps includes the non-constant holomorphic self-maps of a
Riemann surface and for such maps it was shown in \cite{Th2} that the
convolution $C^*$-algebra of the transformation groupoid obtained in
this way is equipped with a one-parameter
group of automorphisms for which the KMS-states correspond to the
conformal measures introduced in complex dynamics by
D. Sullivan. Applied to a particular class of quadratic maps the
result was systems for which the KMS states exhibit phase transition
with spontaneous symmetry breaking in the sense of Bost and
Connes. This illustrates one application of Connes' dictum, obtaining
new models in quantum statistical mechanics. In \cite{Th2} the focus
was on the one-parameter action with its KMS states, and the structure of
the $C^*$-algebras carrying the action was not investigated. It is the purpose of the present paper to present a relatively
detailed description of these $C^*$-algebras $C^*_r(R)$ when they arise from a
rational map $R$ of degree at least two acting on the Riemann sphere $\C$. It
is well-known that the dynamics of such a map is highly complicated,
exhibiting features that are both beautiful and fascinating. As we try
to show here, the structure of the associated $C^*$-algebra is no less fascinating, although the appreciation of it may require a somewhat
more specialized background of the observer than what is needed to
admire the colorful pictures used to depict the dynamics of the maps.

The dynamics of a rational map is partitioned by two totally invariant subsets; the Julia set on which the
map behaves chaotically under iteration and the Fatou set on which its
iterates form an equi-continuous family. As one would expect from
familiarity with crossed products this division gives rise to a
decomposition of $C_r^*(R)$ as an extension where the Fatou set, as the
open subset, gives rise to an ideal $C^*_r(F_R)$ and the Julia set, as the
closed subset, represents the corresponding quotient
$C^*_r(J_R)$. Thus the first decomposition of $C^*_r(R)$ is given by
an extension
\begin{equation*}
\begin{xymatrix}{
0 \ar[r] & \F \ar[r] & C^*_r(R) \ar[r] & \J \ar[r] & 0
}\end{xymatrix}
\end{equation*} 
which reflects the partitioning of the $\C$ by the Julia and Fatou
sets. The
two $C^*$-algebras $C^*_r(J_R)$ and $\F$ in this extension are of very
different nature. The $C^*$-algebra $\J$ of the Julia set is always purely infinite,
nuclear and satisfies the universal coefficient theorem of Rosenberg
and Schochet (\cite{RS}), and it is often, but not always
simple. Ideals in $C^*_r(J_R)$ arise from the possible presence of finite subsets of the Julia
set invariant
under the equivalence relation represented by the transformation groupoid which produces the $C^*$-algebra $C^*_r(R)$; we call this
relation 'restricted orbit equivalence' and it is a relation
which is slightly
stronger than orbit equivalence. The finite subsets of the Julia set invariant under restricted orbit equivalence comprise the finite
subsets considered by Makarov and Smirnov in their work on phase transition in the
thermodynamic formalism associated to the dynamics (\cite{MS1},
\cite{MS2}), and they are closely related to, but not identical with the subsets
introduced in \cite{GPRR} in connection with work on exceptional
rational maps. The possible presence of such subsets of the Julia set
implies that in general the structure of $\J$ must
be decoded from an extension of the form
\begin{equation}\label{jrextint}
\begin{xymatrix}{
0 \ar[r] & C^*_r(J_R \backslash E_R) \ar[r] & \J \ar[r] & B \ar[r] & 0
}\end{xymatrix}
\end{equation}  
where $C^*_r(J_R \backslash E_R)$ is purely infinite and simple, while
$B$ is a finite direct sum of algebras of the form
$M_n(\mathbb C)$ for some $n \leq 3$ or $C(\mathbb T) \otimes
M_n(\mathbb C)$ for some $n \leq 4$.

In contrast to $\J$ the $C^*$-algebra $\F$ of the Fatou set is finite, and
its ideal structure is typically much more complex than that of
$\J$. This is partly due to the fact that the Fatou
set is partitioned into classes of connected components, the
so called stable regions which are termed super-attracting,
attracting, parabolic, Siegel and Herman regions according
to the asymptotic behaviour of their elements under iteration. This
division of $F_R$ results
in a direct sum decomposition of the ideal $C^*_r(F_R)$ where each
direct summand is further decomposed as an extension where the
structure of the ideal depends on the type of the stable region and where the nature of the quotient is
governed by the presence or absence of critical and periodic
points in the region. Specifically,
$$
\F = \oplus_{i=1}^N C^*_r(\Omega_i),
$$
where $C^*_r(\Omega_i)$ is the $C^*$-algebra obtained by restricting
attention to the stable region $\Omega_i \subseteq F_R$. It turns out that the
nature of the algebra $C^*_r(\Omega_i)$ varies with the type of the
stable region:

When $\Omega_i$ is super-attractive there is an
extension
\begin{equation*}
\begin{xymatrix}{
0 \ar[r] & \mathbb K \otimes MT_d \ar[r] & C^*_r(\Omega_i) \ar[r] & B \ar[r] & 0
}\end{xymatrix}
\end{equation*} 
where $B$ is a finite direct sum of algebras stably isomorphic to
either $\mathbb C$ or the continuous functions on the Cantor set. The
algebra $MT_d$ is the mapping torus
of an endomorphism on a Bunce-Deddens algebra of type $d^{\infty}$
where $d$ is the product of the valencies of the elements in the
critical orbit.

When $\Omega_i$ is attractive there is a
an
extension
\begin{equation*}
\begin{xymatrix}{
0 \ar[r] & \mathbb K \otimes C\left(\mathbb
  T^2\right) \ar[r] & C^*_r(\Omega_i) \ar[r] & B \ar[r] & 0
}\end{xymatrix}
\end{equation*} 
where $B$ is a finite direct sum of algebras stably isomorphic to
either $\mathbb C$ or the continuous functions on the circle $\mathbb
T$.

When $\Omega_i$ is parabolic there is a
an
extension
\begin{equation*}
\begin{xymatrix}{
0 \ar[r] & \mathbb K \otimes C\left(\mathbb
  T\right) \otimes C_0(\mathbb R) \ar[r] & C^*_r(\Omega_i) \ar[r] & B \ar[r] & 0
}\end{xymatrix}
\end{equation*} 
where $B$ is a finite direct sum of algebras stably isomorphic to $\mathbb C$.

When $\Omega_i$ is of Siegel type there is a
an
extension
\begin{equation}\label{S-H}
\begin{xymatrix}{
0 \ar[r] & \mathbb K \otimes C_0\left( \mathbb R\right) \otimes A_{\theta} \ar[r] & C^*_r(\Omega_i) \ar[r] & B \ar[r] & 0
}\end{xymatrix}
\end{equation} 
where $B$ is a finite direct sum of algebras stably isomorphic to
either $\mathbb C$ or the continuous functions on the circle $\mathbb
T$, and $A_{\theta}$ is the irrational rotation algebra corresponding
to the angle of rotation in the Siegel domain inside $\Omega_i$.

Finally, when $\Omega_i$ is of Herman type there is a
an extension quite similar to (\ref{S-H}). The only difference is that
while the quotient algebra $B$ must contain a summand stably
isomorphic to $C(\mathbb T)$ in the Siegel case, in the Herman case all summands are
stably isomorphic to $\mathbb C$.

It almost goes without saying that the entire structure in the
decomposition of $C^*_r(R)$ described above reflects easily identified
structures of the dynamics of $R$. For example, the difference between
the structure of the summands in $\F$ coming from a Siegel and Herman
region is due to the periodic point in a Siegel domain which is absent in
a Herman ring.

From the point of view of operator algebra theory a study of a
non-simple $C^*$-algebra often begins with a description of the primitive
ideals and the corresponding irreducible quotients. In \cite{CT} T. Carlsen and the author identified the primitive and
maximal ideals of the $C^*$-algebras arising from the transformation
groupoid of a locally injective surjection on a finite dimensional
compact metric space. In the final section of the present paper the method
from \cite{CT} is carried over to the groupoid $C^*$-algebras of
rational maps and we obtain in this way a description of the primitive ideals and primitive
quotients. In particular, it is shown that the primitive ideal space
of $C^*_r(R)$ is only Hausdorff in the hull-kernel topology when it
has to be, i.e. when $C^*_r(R)$ is simple. This occurs only when $J_R
= \C$ and there are no finite sets invariant under restricted orbit
equivalence. In all other cases the primitive ideal space is not even
$T_0$. 

While there is often a rich variety of primitive quotients, there are
always very few types of simple quotients. The finite invariant subsets under
restricted orbit equivalence give rise to maximal ideals, but the
corresponding simple quotients are matrix algebras of size no more than
4. In most cases $\J$ is also a simple quotient, but only when there
are no finite subsets of $J_R$ invariant under restricted
orbit equivalence. There are no other simple quotients. In particular,
when there are finite subsets of $J_R$ invariant under restricted
orbit equivalence the only simple quotients of $C^*_r(R)$ are matrix
algebras of size not exceeding 4.

There are other ways to associate a $C^*$-algebra to a rational map,
and we refer to \cite{DM} and \cite{KW} for these. It would be
interesting to find the precise relationship between the algebras
investigated here and those of Kajiwara and Watatani. Presently it is
only clear that they are generally very different.

\smallskip

\emph{Acknowledgement.} This work was completed during a visit to the
Institut Henri Poincar\'e as part of the Research in Paris program, and I take the opportunity to thank
the IHP for support and for the exceptional working conditions.

\section{\'Etale groupoids and $C^*$-algebras from dynamical systems}\label{dynsys}

Let $G$ be an \'etale second countable locally compact Hausdorff groupoid with unit
space $G^{(0)}$. Let $r : G \to G^{(0)}$ and $s : G \to G^{(0)}$ be
the range and source maps, respectively. For $x \in G^{(0)}$ put $G^x = r^{-1}(x), \ G_x = s^{-1}(x) \ \text{and} \ \Is_x =
s^{-1}(x)\cap r^{-1}(x)$. Note that $\Is_x$ is a group, the \emph{isotropy group} at $x$. The space $C_c(G)$ of continuous compactly supported
functions is a $*$-algebra when the product is defined by
$$
(f_1 f_2)(g) = \sum_{h \in G^{r(g)}} f_1(h)f_2(h^{-1}g)
$$
and the involution by $f^*(g) = \overline{f\left(g^{-1}\right)}$. Let $x\in G^{(0)}$. There is a
representation $\pi_x$ of $C_c(G)$ on the Hilbert space $l^2(G_x)$ of
square-summable functions on $G_x$ given by 
\begin{equation*}\label{pix}
\pi_x(f)\psi(g) = \sum_{h \in G^{r(g)}} f(h)\psi(h^{-1}g) .
\end{equation*}
The \emph{reduced groupoid
  $C^*$-algebra} $G^*_r(G)$ is the completion of $C_c(G)$ with respect to the norm
$$
\left\|f\right\| = \sup_{x \in G^{(0)}} \left\|\pi_x(f)\right\| .
$$

\subsubsection{Stability of $C^*_r(G)$}

We shall need the following sufficient condition for stability of
$C^*_r(G)$. Recall that a $C^*$-algebra $A$ is \emph{stable} when $A
\otimes \mathbb K  \simeq A$ where $\mathbb K$ denotes the
$C^*$-algebra of compact operators on an infinite dimensional
separable Hilbert space. Recall also that a \emph{bi-section} in $G$ is
an open subset $U \subseteq G$ such that $r : U \to G^{(0)}$ and $s :
U \to G^{(0)}$ both are injective.

\begin{lemma}\label{stability} Let $G$ be a locally compact second
  countable \'etale groupoid. Assume that for
  every compact subset $K \subseteq G^{(0)}$ there is a finite
  collection $\left\{U_i\right\}_{i=1}^N$ of bi-sections in $G$ such
  that
\begin{enumerate}
\item[i)] $K \subseteq \bigcup_{i=1}^N s(U_i)$,
\item[ii)] $r(U_i) \cap r(U_j) = \emptyset$ when $i\neq j$, and
 \item[iii)] $K \cap \bigcup_{i=1}^N r\left(U_i\right) = \emptyset$.
\end{enumerate}
It follows that $C^*_r(G)$ is a
  stable $C^*$-algebra.
\end{lemma}
\begin{proof} By Theorem 2.1 and b) of Proposition 2.1 in
  \cite{HR} it suffices to consider a positive element $a \in
  C^*_r(G)$ and an $\epsilon > 0$ and show that there is an element $v
  \in C^*_r(G)$ such that $\left\|v^*v -a\right\| \leq \epsilon$ and
  $v^2 = 0$. Write $a = a_0^*a_0$ for some $a_0 \in C^*_r(G)$. By approximating
$a_0$ with an element $h \in C_c(G)$ and taking $b = h^*h$ we obtain
an element $b \in C_c(G)$ which is positive in $C^*_r(G)$ and
satisfies that $\left\|a-b\right\| \leq \epsilon$. Let $\supp b$ be
the support of $b$ in $G$ and set $K = r\left(
  \supp b\right)$. By assumption there is a finite collection
$\left\{U_i\right\}_{i=1}^N$ of bi-sections such that i)-iii)
hold. Let $\left\{h_i\right\} \subseteq C_c\left(G^{(0)}\right)$ be a
partition of unity on $K$ subordinate to $\left\{s\left(U_i\right)\right\}_{i=1}^N$. For each $i$ let $f_i \in C_c(G)$ be
supported in $U_i$ and satisfy that $f_i=1$ on $s^{-1}\left(\supp
h_i\right) \cap U_i$. Set $w = \sum_{i=1}^{N} f_i\sqrt{h_i} \in C_c(G)$
and note that $w^*wb = \sum_i h_ib = b$ while $bw= 0$. Set $v =
w\sqrt{b}$.
\end{proof}

\subsubsection{$G$-orbits and reductions} When $W$ is a subset of $G^{(0)}$ we set
$$
G_W = \left\{ g \in G : \ r(g), s(g) \in W \right\},
$$
which is a subgroupoid of $G$, called a \emph{reduction} of $G$. When
$W$ is an open subset of $G$ the reduction $G_W$ will be an \'etale
groupoid in the relative topology inherited from $G$ and there is an embedding $C_r^*(G_W) \subseteq
C^*_r(G)$, cf. e.g. Proposition 1.9 in \cite{Ph}. In fact,
$C^*_r\left(G_W\right)$ is a hereditary $C^*$-subalgebra of
$C^*_r\left(G\right)$. When $x \in G^{(0)}$, the set $Gx = \left\{ r(g) : \ g \in G_x
\right\}$ will be called the \emph{$G$-orbit} of $x$. We say that $W \subseteq G^{(0)}$ is \emph{$G$-invariant}
when $x \in W \Rightarrow Gx \subseteq W$. When $W$ is $G$-invariant and locally compact in the topology
inherited from $G^{(0)}$ the reduction $G_W$ is an \'etale locally
compact groupoid in the topology inherited
from $G$. If $W$ is both open and
$G$-invariant $C^*_r\left(G_W\right)$ is an ideal in
$C^*_r\left(G\right)$. Similarly, when $F$ is a closed subset of
$G^{(0)}$ which is also $G$-invariant, $C^*_r\left(G_F\right)$ is a quotient of
$C^*_r\left(G\right)$. It is known that under a suitable amenability condition the
kernel of the quotient map 
$$
\pi_F : C^*_r\left(G\right) \to
C^*_r\left(G_F\right)
$$ 
is $C^*_r\left(G_{G^{(0)} \backslash
    F}\right)$. We shall avoid the amenability issue here and prove
this equality directly in the cases we are interested in. See Section
\ref{julia/fatou}.

We shall need the following fact which follows straightforwardly from
the definitions.

\begin{lemma}\label{basdecomp} Assume that there is a finite partition
  $G^{(0)} = \sqcup_{i=1}^n W_i$ such that each $W_i$ is open and
  $G$-invariant. It follows that
$$
C^*_r(G) \simeq \oplus_{i=1}^n C^*_r\left(G_{W_i}\right) .
$$
\end{lemma}

The following result was obtained by Muhly, Renault and Williams in \cite{MRW}.  

\begin{thm}\label{MRW} Let $G$ be an \'etale second
  countable locally compact Hausdorff groupoid and $W \subseteq G^{(0)}$ an open subset
  such that $Gx \cap W \neq \emptyset$ for all $x \in G^{(0)}$. It
  follows that $C^*_r(G)$ is stably isomorphic to
  $C^*_r\left(G_W\right)$.
\end{thm}
\begin{proof} The set $\bigcup_{x \in W} G_x$ is a
  $\left(G,G_W\right)$-equivalence in the sense of \cite{MRW} and
  hence Theorem 2.8 of \cite{MRW} applies.
\end{proof}

\subsubsection{Pure infiniteness of $C^*_r(G)$}\label{pure13}

Following \cite{An} we say that an \'etale groupoid $G$ is
\emph{essentially free} when the points $x$ of the unit space
$G^{(0)}$ for which the isotropy group $\Is_x$ is
trivial (i.e. only consists of $\{x\}$) is dense in $G^{(0)}$. 
In the same vein we say that $G$ is \emph{locally contracting} when every
open non-empty subset of $G^{(0)}$ contains an open subset
$V$ with the property that there is an open bisection $S$ in $G$
such that $\overline{V} \subseteq s(S)$ and
$\alpha_{S}^{-1}\left(\overline{V}\right) \subsetneq V$, when $\alpha_S
: r(S) \to s(S)$ is the homeomorphism defined such that $\alpha_S(x) =
s(g)$ where $g \in S$ is the unique element with $r(g) =x$, cf. Definition
2.1 of \cite{An} (but note that the source map is denoted by $d$ in
\cite{An}). 

We say that a $C^*$-algebra is \emph{purely infinite} when every
non-zero hereditary $C^*$-subalgebra of $A$ contains an infinite projection.
Proposition 2.4 of \cite{An} then says the following.

\begin{thm}\label{an}
Let $G$ be an \'etale second countable locally compact Hausdorff
groupoid. Assume that $G$ is essentially free and locally
contracting. Then $C^*_r(G)$ is purely infinite.
\end{thm}

\subsection{The transformation groupoid of a local homeomorphism}\label{trans}

In this section we describe the construction of an \'etale second countable locally compact
Hausdorff groupoid from a local
homeomorphism of a locally compact Hausdorff space which was introduced in increasing generality by
J. Renault \cite{Re}, V. Deaconu \cite{De} and Anantharaman-Delaroche
\cite{An}.

Let $X$ be a second countable locally compact Hausdorff space and
$\varphi : X \to X$ a local homeomorphism. Thus we assume that
$\varphi$ is open and locally injective, but not necessarily
surjective. Set
$$
G_{\varphi} = \left\{ (x,k,y) \in X \times \mathbb Z  \times X :
  \ \exists n,m \in \mathbb N, \ k = n -m , \ \varphi^n(x) =
  \varphi^m(y)\right\} .
$$
This is a groupoid with the set of composable pairs being
$$
G_{\varphi}^{(2)} \ =  \ \left\{\left((x,k,y), (x',k',y')\right) \in G_{\varphi} \times
  G_{\varphi} : \ y = x'\right\}.
$$
The multiplication and inversion are given by 
$$
(x,k,y)(y,k',y') = (x,k+k',y') \ \text{and}  \ (x,k,y)^{-1} = (y,-k,x)
.
$$
Note that the unit space of $G_{\varphi}$ can be identified with
$X$ via the map $x \mapsto (x,0,x)$. Under this identification the
range map $r: G_{\varphi} \to X$ is the projection $r(x,k,y) = x$
and the source map the projection $s(x,k,y) = y$.

To turn $G_{\varphi}$ into a locally compact topological groupoid, fix $k \in \mathbb Z$. For each $n \in \mathbb N$ such that
$n+k \geq 0$, set
$$
G_{\varphi}(k,n) = \left\{ \left(x,l, y\right) \in X \times \mathbb
  Z \times X: \ l =k, \ \varphi^{k+n}(x) = \varphi^n(y) \right\} .
$$
This is a closed subset of the topological product $X \times \mathbb Z
\times X$ and hence a locally compact Hausdorff space in the relative
topology.
Since $\varphi$ is locally injective $G_{\varphi}(k,n)$ is an open subset of
$G_{\varphi}(k,n+1)$ and hence the union
$$
{G_{\varphi}}(k) = \bigcup_{n \geq -k} {G_{\varphi}}(k,n) 
$$
is a locally compact Hausdorff space in the inductive limit topology. The disjoint union
$$
G_{\varphi} = \bigcup_{k \in \mathbb Z} {G_{\varphi}}(k)
$$
is then a locally compact Hausdorff space in the topology where each
$G_{\varphi}(k)$ is an open and closed set. In fact, as is easily verified, $G_{\varphi}$ is an \'etale groupoid,
i.e. the range and source maps are local homeomorphisms. The groupoid
$G_{\varphi}$ will be called the \emph{transformation groupoid} of
$\varphi$. To simplify notation we denote in the following the
corresponding $C^*$-algebra by $C^*_r(\varphi)$; i.e. we set
$$
C^*_r\left(G_{\varphi}\right) = C^*_r(\varphi).
$$ 
Note that the $G_{\varphi}$-orbit $G_{\varphi}x$ of a point $x
\in X$ is the \emph{orbit} of $x$ under $\varphi$, i.e.
$$
G_{\varphi}x = \bigcup_{n,m \in \mathbb N}
\varphi^{-n}\left(\varphi^m(x)\right);
$$ 
by some authors called the full or grand orbit to
distinguish it from $\left\{\varphi^n(x) : \ n \in \mathbb N\right\}$,
which we will call the \emph{forward orbit}. 

When the local homeomorphism $\varphi$ is proper, in the sense that
inverse images of compact sets are compact, the $C^*$-algebra
$C^*_r(\varphi)$ can be realised as the crossed product by an
endomorphism, cf. \cite{De},\cite{An}, in the following way: The
subset
$$
R_{\varphi} = G_{\varphi}(0) \simeq \left\{ (x,y) \in X \times X : \ \varphi^n(x) =
  \varphi^n(y) \ \text{for some} \ n \in \mathbb N \right\}
$$
is an open subgroupoid of $G_{\varphi}$ and an \'etale groupoid in
itself. The reduced groupoid $C^*$-algebra
$C^*_r\left(R_{\varphi}\right)$ is a $C^*$-subalgebra of
$C^*_r(\varphi)$. When $\varphi$ is proper there is an endomorphism $\widehat{\varphi}
: C^*_r(R_{\varphi}) \to C^*_r(R_{ \varphi})$ defined such that
$$
\widehat{\varphi}(f)(x,y) = \left( \# \varphi^{-1}(\varphi(x)) \#
  \varphi^{-1}(\varphi(y))\right)^{-1/2}f(\varphi(x),\varphi(y))
$$
when $f \in C_c(R_{\varphi})$. We will refer to this endomorphism as
\emph{the Deaconu endomorphism}. As shown in \cite{An} there is an
isomorphism
$$
C^*_r(\varphi) = C^*_r\left(R_{\varphi}\right)
\rtimes_{\widehat{\varphi}} \mathbb N,
$$
where the crossed product is a crossed product by an endomorphism both
in the sense of Paschke (\cite{P}) and the sense of Stacey (\cite{St}).

\subsection{The transformation groupoid of a rational map}

In this section we describe an \'etale groupoid coming from a non-constant
holomorphic map on a Riemann surface by a construction introduced in
\cite{Th2}. But since we shall focus on the Riemann sphere in this
paper we restrict the presentation accordingly.

Let $\overline{\mathbb C}$ be the Riemann sphere and $R :
\overline{\mathbb C} \to \overline{\mathbb C}$ a rational map of degree
at least $2$. Consider a subset $X \subseteq \overline{\mathbb C}$ which
is locally compact in the topology inherited from $\overline{\mathbb
  C}$, without isolated points and \emph{totally
$R$-invariant} in the sense that $R^{-1}(X) = X$. Let
$\mathcal P$ be the pseudo-group on $X$ of local homeomorphisms $\xi : U
\to V$ between open subsets of $X$ with the property that there are
open subsets $U_1,V_1$ in $\overline{\mathbb C}$ and a bi-holomorphic map $\xi_1 : U_1
\to V_1$ such that $U_1 \cap X = U, \ V_1 \cap X = V$ and $\xi =
\xi_1$ on $U$. For each $k \in \mathbb Z$ we denote by
$\mathcal T_k(X)$ the elements $\eta : U \to V$ of $\mathcal P$ with the
property that there are natural numbers $n,m$ such that $k = n-m$ and
\begin{equation}\label{crux0} 
R^n(z) = R^m(\eta(z))  \ \ \forall z \in U.
\end{equation}
The elements of $\mathcal T =\bigcup_{k \in \mathbb Z} \mathcal T_k(X)$ will
be called \emph{local transfers} for $R|_X$. 
We denote by $[\eta]_x$ the germ at a point $x \in X$ of an element $\eta \in \mathcal
T$. Set
$$
\mathcal G_{X} = \left\{ (x,k,\eta,y) \in X \times \mathbb Z \times
  \mathcal P  \times X : \ \eta \in \mathcal T_k(X) , \ \eta(x)
  = y \right\} .
$$
We define an equivalence relation $\sim$ in $\mathcal G_{X}$ such
that $(x,k,\eta,y) \sim (x',k',\eta',y')$ when $ x = x', \ y = y', \ k
= k'$ and $[\eta]_x = [\eta']_x$. Let
$\left[x,k,\eta,y\right]$ denote the equivalence class of $(x,k,\eta,y)  \in \mathcal G_{X}$. The quotient space 
$G_{X} = \mathcal G_{X}/{\sim}$ is a groupoid where two elements
$\left[x,k,\eta,y\right]$ and $\left[x',k',\eta',y'\right]$ are
composable when $y= x'$ and their product is
$$
\left[x,k,\eta,y\right]\left[y,k',\eta',y'\right] =
\left[x,k+k',\eta'\circ \eta ,y'\right] .
$$ 
The inversion in $G_{X}$ is defined such that
$\left[x,k,\eta,y\right]^{-1} = \left[ y,-k,\eta^{-1},x \right]$. The unit space of $G_{X}$ can be identified with $X$ via the map $x \mapsto [x,0,\id,x]$, where $\id$ is the
identity map on $X$. When $\eta \in \mathcal T_k(X)$ and $U$ is an open subset of the domain
of $\eta$ we set
\begin{equation}\label{baseset}
U(\eta) = \left\{ \left[z,k,\eta,\eta(z)\right]  : \ z \in U \right\}.
\end{equation}
It is straightforward to verify that by varying $k$, $\eta$ and $U$, the sets (\ref{baseset}) constitute a base
for a topology on $G_{X}$. A crucial point is that the topology is
also Hausdorff since the local transfers are holomorphic. It follows that $G_{X}$ is an 
  \'etale second countable locally compact Hausdorff groupoid,
  cf. \cite{Th2}. 

It is the possible presence of critical points of $R$ in $X$ which
prevents us from using the procedure of Section \ref{trans} to get an
\'etale groupoid. The additional feature, the local transfers, which
is introduced to obtain a well-behaved \'etale groupoid is therefore
not necessary when there are no critical points in $X$, and it is
therefore reassuring that the two constructions coincide in the
absence of critical points in $X$.

\begin{prop}\label{nocrit} Assume that $Y \subseteq X$ is an open
  subset of $X$ which does not contain any
  critical point of $R$ and is $R$-invariant in the sense that $R(Y)
  \subseteq Y$. Then the reduction $\left(G_X\right)_Y$ is isomorphic,
  as a topological groupoid, to the transformation groupoid of the local homeomorphism $R|_Y
  : Y \to Y$. 
\end{prop}
\begin{proof} Define $\mu : \left(G_X\right)_Y
  \to G_{R|_Y}$ such that $\mu[x,k,\eta,y] = (x,k,y)$.

$\mu$ is surjective:
  If $(x,k,y) \in G_{R|_Y}$ there are $n,m \in \mathbb N$ such that $k
  = n-m$ and $R^n(x) = R^m(y)$. Note that $(R^n)'(x) \neq 0$ and
  $(R^m)'(y) \neq 0$ since $Y$ does not contain critical points. It
  follows that there are open neighbourhoods $U$ and $V$ of $x$
  and $y$ in $\overline{\mathbb C}$,
  respectively, such that $R^n : U \to R^n(U)$ and $R^m : V \to
  R^m(V)$ are univalent holomorphic maps. Let $R^{-m} : R^m(V) \to V$ be
  the inverse of $R^m : V \to
  R^m(V)$. Then $\eta_0 = R^{-m} \circ R^n : R^{-n}\left(R^m(V)\right)
  \cap U \to R^{-m}(R^n(U)) \cap V$ is bi-holomorphic. Set $\eta =
  \eta_0|_X$ and note that $[x,k,\eta,y] \in \left(G_X\right)_Y$.

$\mu$ is injective: Assume that $\mu[x,k,\eta,y] =
\mu[x',k',\eta',y']$. Then $(x,k,y) = (x',k',y')$. Since
$\left(R^m\right)'(y) \neq 0$ it follows that $R^m$ is injective in a
neighborhood of $y$. Since $R^m(\eta(z)) = R^n(z) =
R^m\left(\eta'(z)\right)$ for all $z$ close to $x$ it follows that
$\eta = \eta'$ in a neighborhood of $x$, i.e. $\left[x,k,\eta,y\right]
= \left[x',k',\eta',y'\right]$. 

%To
%this end note that there is
%an open neighborhood $U_1$ of $x$ in $\C$ and natural numbers $n,m$ such that $R^n(z) =
%R^m(\eta(z))$ for all $z \in U_1$, as well as an open neighborhood $U_2$ of $x$ in $\C$ such that $R^n(z) =
%R^m(\eta'(z))$ for all $z \in U_2$. Since $(R^n)'(x) \neq 0$ and
%  $(R^m)'(y) \neq 0$ there are open neighborhoods $U \subseteq U_1
%  \cap U_2$ and $V$ of $x$
%  and $y$ in $\C$,
%  respectively, such that $R^n : U \to R^n(U)$ and $R^m : V \to
%  R^m(V)$ are univalent holomorphic maps and $R^n(U) \subseteq R^m(V)$. Let $R^{-m} : R^m(V) \to V$ be
%  the inverse of $R^m : V \to
%  R^m(V)$ and note that $\eta(z) = R^{-m}\circ R^n(z) = \eta'(z)$ for
%  all $z \in U$.

$\mu$ is continuous: Let $\mu[x,k,\eta,y] = (x,k,y)$ and consider an
open neighborhood $\Omega$ of $(x,k,y)$ in $G_{R|_Y}$. There
is then an $m \in \mathbb N$ and a local transfer $\eta \in \mathcal T_k\left(X\right)$ such that $k+ m \geq 0$, $\eta(x) = y$ and $R^{k+m}(z) =
R^m(\eta(z))$ for all $z$ in a neighborhood of $x$. By definition of
the topology of $G_{R|_Y}$, the set $\Omega \cap G_{R|_Y}(k,m)$ is
open in $G_{R|_Y}(k,m)$ when this set has the relative topology inherited from $X \times \mathbb Z
\times X$. From this fact it follows that there is an open
neighborhood $W$ of $x$ in the domain of $\eta$ such that
$(z,k,\eta(z)) \in \Omega \cap G_{R|_Y}(k,m)$
for all $z \in W$. Since $Y$ is open in $X$ we can shrink $W$ to ensure
that $W \subseteq Y$ and $\eta(W) \subseteq Y$. Then
$\left\{[z,k,\eta,\eta(z)] : \ z \in W \right\}$ is an open
neighborhood of $[x,k,\eta ,y]$ in $\left(G_X\right)_Y$ such that $\mu\left( \left\{[z,k,\eta,\eta(z)] : \ z \in W \right\}\right)
\subseteq \Omega$.

$\mu$ is open: Let $\eta \in \mathcal T_k(X)$ be a local
transfer. Let $U$ be an open subset of the domain of $\eta$ such that
$U \subseteq Y$ and $\eta(U) \subseteq Y$. It suffices to show that
$$
\eta \left( \left\{ [z,k,\eta,\eta(z)] : \ z \in U \right\} \right)
$$
is open in $G_{R|_Y}$. To this end note that there is an $m \in \mathbb N$ such that
$k+m \geq 0$ and $R^{k+m}(z) = R^m(\eta(z))$ for all $z \in
U$. Consider a point $(z_1,k,z_2) \in \eta \left( \left\{
    [z,k,\eta,\eta(z)] : \ z \in U \right\} \right)$. Then
$(z_1,k,z_2) \in G_{R|_Y}(k,m)$. Let $W$ be an open neighborhood
of $z_2$ in $X$ such that $W \subseteq \eta(U)$ and $R^m$ is injective
on $W$. If $(z_1',k,z_2') \in \left( \eta^{-1}(W) \times \{k\} \times
  W\right) \cap G_{R|_Y}(k,m)$ we have that $R^m(z'_2) =
R^{k+m}(z'_1) = R^m(\eta(z'_1))$ which implies that $z'_2 =
\eta(z'_1)$. This shows that 
 $$
\left( \eta^{-1}(W) \times \{k\} \times
  W \right) \cap G_{R|_Y}(k,m) \subseteq \eta \left( \left\{
    [z,k,\eta,\eta(z)] : \ z \in U \right\} \right) .
$$
\end{proof}

When we take the set $X$ in the construction of $G_X$ to be the whole Riemann
sphere $\overline{\mathbb C}$ we obtain the groupoid
$G_{\overline{\mathbb C}}$ and the corresponding $C^*$-algebra
$C^*_r\left(G_{\overline{\mathbb C}}\right)$. There is no good
reason to emphasize $\overline{\mathbb C}$ so we will denote the
groupoid $G_{\overline{\mathbb C}}$ by $G_R$ instead and 
$C^*_r\left(G_{\overline{\mathbb C}}\right)$ by $C^*_r(R)$. We call
$G_R$ the \emph{transformation groupoid} of $R$. It
  follows from Lemma 4.1 in \cite{Th2} that two elements $x,y \in
  \overline{\mathbb C}$
  are in the same $G_R$-orbit if and only if there are natural numbers
  $n,m \in \mathbb N$ such that $R^n(x) = R^m(y)$ and $\val
  \left(R^n,x\right) = \val \left(R^m,y\right)$, where
  $\val\left(R^n,x\right)$ is the valency of $R^n$ at the point
  $x$. In particular, the $G_R$-orbit of $x$, which we will denote by
  $\Ro(x)$ in the following, is a subset of
  the orbit of $x$ conditioned by an equality of valencies. Specifically,
\begin{equation*}
\begin{split}
&\Ro(x) = \\
&\left\{ y \in \C : \ R^n(y) = R^m(x) \ \text{and} \
  \val\left(R^n,y\right) = \val\left(R^m,x\right) \ \text{for some} \
  n,m \in \mathbb N \right\}.
\end{split}
\end{equation*}
We call $\Ro(x)$ the \emph{restricted orbit} of $x$. In the following a subset $Y$ of $\C$ will be called \emph{restricted
  orbit invariant}, or $\Ro$-invariant, when $y \in Y \Rightarrow
\Ro(y) \subseteq Y$. Note that a totally $R$-invariant
subset is $\Ro$-invariant, but the converse can fail when there are
isolated points in $Y$.

We shall work with many different reductions of $G_R$. Recall that for
any subset $Y \subseteq \C$ the \emph{reduction} of $G_R$ to $Y$ is
the set
$$
G_Y = \left\{ \left[x,k,\eta,y\right] \in G_R : \ x,y \in Y \right\} .
$$
This is always a subgroupoid of $G_R$. If $X$ is both locally compact in
the topology inherited from $\C$ and $\Ro$-invariant, $G_X$ will be a
locally compact Hausdorff \'etale groupoid in the topology inherited
from $G_R$. When $X$ is also totally $R$-invariant and has no isolated
points the reduction is equal to the groupoid $G_X$ described above,
showing that our notation is consistent. Furthermore, when $X$ is open in the
relative topology of a subset $Y \subseteq \C$ which is
$\Ro$-invariant and locally compact in the topology inherited from
$\C$, the reduction $G_X$ is again a
locally compact Hausdorff \'etale groupoid in the topology inherited
from $G_R$. The reduction $G_X$ of $G_R$ by such a set is the most general type
of reduction we shall need in this paper. To simplify notation we set
$$
C^*_r(X) = C^*_r\left(G_X\right) .
$$

\section{The Julia-Fatou extension}\label{julia/fatou}

%To study the 
%$C^*$-algebra $C^*_r(R)$ the first step will be to decompose it as an extension
%of $C^*_r\left({J_R}\right)$ by $C^*_r\left({F_R}\right)$, where $J_R$ and
%$F_R$ are the Julia and Fatou sets of $R$, respectively. 

Let $X$ be an $\Ro$-invariant subset of $\C$ which is locally compact
in the relative topology, or at least an open subset of such a set. The $C^*$-algebra $C^*_r\left({X}\right)$ carries a natural
action $\beta$ by the circle group $\mathbb T$ defined such that
\begin{equation}\label{gaugeact}
\beta_{\mu}(f)\left[x,k,\eta, y\right] = \mu^k
f\left[x,k,\eta, y\right] 
\end{equation}
when $f \in C_c\left(G_{X}\right)$ and $\mu \in \mathbb T$. To
identify the fixed point algebra of $\beta$ set
$$
G^0_X = \left\{ [x,k,\eta,y]\in G_X: \ k = 0 \right\} .
$$
Since $G^0_X$ is an open subgroupoid of $G_X$ there is an inclusion
$C_c\left(G^0_X\right) \subseteq
C_c\left(G_X\right)$ which extends to an embedding
$C^*_r\left(G^0_X\right) \subseteq C^*_r\left(X\right)$
of $C^*$-algebras.

\begin{lemma}\label{fixalg} $C^*_r\left(G^0_X\right)$ is the fixed
  point algebra $C^*_r(X)^{\beta}$ of the gauge action.
\end{lemma}
\begin{proof} Let $ a \in C^*_r\left({X}\right)^{\beta}$. For any
  $\epsilon > 0$ there is a function $f \in C_c\left(G_{X}\right)$
  such that $\left\|a -f \right\| \leq \epsilon$. We can then write
  $f$ as a finite sum $f
  = \sum_{k \in \mathbb Z} f_k$ where $f_k$ is supported in
$$
\left\{ \left[x,l,\eta,y\right] \in G_X : \ l = k \right\}.
$$
Since $\int_{\mathbb T} \beta_{\mu}(f_k) \ d \mu = \int_{\mathbb T}
\mu^k f_k \ d\mu = 0$ when $k \neq
  0$ and $\left\| a - \int_{\mathbb T} \beta_{\mu}(f) \ d
    \mu\right\| \leq \epsilon$ we deduce that $\left\| a -
    f_0\right\| \leq \epsilon$. This shows that $C^*_r\left({X}\right)^{\beta}
  \subseteq C^*_r\left(G^0_X\right)$, and the reversed inclusion is trivial. 
\end{proof}

Just as for local homeomorphisms, \cite{An}, \cite{De}, there is an
inductive limit decomposition of $C^*_r\left(G^0_{X}\right)$ which
throws some light on its structure. Let $n \in \mathbb N$. Set
$$
G^0_{X}(n) = \left\{ \left[ x,0, \eta,y\right] \in G_{X} : \
  R^n(z) = R^n\left(\eta(z) \right) \ \text{in a neighbourhood of
    $x$} \right\} .
$$
Each $G^0_{X}(n)$ is an open subgroupoid of $G^0_{X}$,
$G^0_X(n) \subseteq G^0_{X}(n+1)$ for all $n$ and
$G^0_{X} = \bigcup_n G^0_{X}(n)$.
It follows that $C^*_r\left(G^0_{X}(n)\right) \subseteq
C^*_r\left(G^0_{X}(n+1)\right) \subseteq
C^*_r\left(G^0_{X}\right)$ for all $n$, and
\begin{equation}\label{union}
C^*_r\left(G^0_{X}\right) = \overline{\bigcup_n
  C^*_r\left(G^0_{X}(n)\right)}.
\end{equation}

Assume now that $X$ is an $\Ro$-invariant subset of $\C$ which is locally compact
in the relative topology, and not just an open subset of such a set. Let $Y$ be a closed $\Ro$-invariant subset of $X$. Then $X\backslash
Y$ is open in $X$ and $\Ro$-invariant. Since $Y$ and $X \backslash Y$ are
locally compact in the topology inherited from $\C$ we can consider
the reduced groupoid $C^*$-algebras $C^*_r(Y)$ and $C^*_r\left(X
  \backslash Y\right)$. Furthermore, we have a surjective $*$-homomorphism
$$
\pi_Y : C^*_r\left(X\right) \to C^*_r\left(Y\right)
$$ 
because $Y$ is closed and $\Ro$-invariant in $X$.

\begin{lemma}\label{quot} Let $X$ be a $\Ro$-invariant subset of $\C$
  which is locally compact in the relative topology. Let $Y$ be a closed $\Ro$-invariant subset
  of $X$. The sequence
\begin{equation*}
\begin{xymatrix}{
0 \ar[r] & C^*_r\left({X \backslash
    Y}\right) \ar[r] & C^*_r\left(X\right) \ar[r]^-{\pi_Y} &
C^*_r\left(Y\right) \ar[r] & 0
}\end{xymatrix}
\end{equation*}
is exact.
\end{lemma}
\begin{proof} It is clear that
$C^*_r\left({X \backslash
    Y}\right) \subseteq \ker \pi_Y$. To establish the reverse
inclusion, let $a \in \ker \pi_Y$ and let $\epsilon > 0$. Note
  first that the formula (\ref{gaugeact}) also defines an action by
  $\mathbb T$ on $C^*_r\left(Y\right)$ which makes
  $\pi_Y$ equivariant. It follows that $\ker \pi_Y$ is left globally
  invariant by the gauge action, and there is therefore
  an approximate unit in $\ker \pi_Y$ consisting of elements fixed by
  the gauge action. It follows then from Lemma \ref{fixalg} that there is an element $u \in
  C^*_r\left(G^0_X\right) \cap \ker \pi_Y$ such that
  $\left\|ua-a\right\| \leq \epsilon$. It follows from (\ref{union}) that
$$
C^*_r\left(G^0_X\right) \cap \ker \pi_Y = \overline{ \bigcup_n
  C^*_r\left(G^0_X(n)\right) \cap \ker \pi_Y} ,
$$
and there is therefore an $n \in \mathbb N$ and an element $v \in
C^*_r\left(G^0_X(n)\right) \cap \ker \pi_Y$ such that
$\left\|ua-va\right\| \leq \epsilon$. Note now that $\# R^{-n}(x)\leq
(\deg R)^n$ for all $x \in \C$ when
$\deg R$ is the degree of $R$. By definition of the norm on
$C^*_r\left(G^0_X(n)\right)$ this gives us the estimate
\begin{equation}\label{Kest}
\left\|f\right\| \leq (\deg R)^n \sup_{\gamma \in G^0_X(n)}
\left|f(\gamma)\right|
\end{equation}
for all $f \in C_c\left(G^0_X(n)\right)$. Set $\delta
={\epsilon}((2(\deg R)^n+1)(\|a\| +1))^{-1}$ and choose $g \in C_c(G^0_X(n))$ such that $\left\|g - v
\right\| \leq \delta$. Then $\left\|\pi_Y(g)\right\| \leq \delta$ and
hence 
$$
\sup_{\gamma \in G^0_X(n) \cap s^{-1}(Y)} \left|g(\gamma)\right|
\leq \delta,
$$
by Proposition II 4.2 of \cite{Re}. We can therefore write $g = g_1 +
g_2$ where $g_1 \in C_c(G^0_X(n))$ has support in $s^{-1}(X\backslash Y)$ and 
$\sup_{\gamma \in G^0_X(n)} \left|g_2(\gamma)\right| \leq 2
\delta$. 
It follows then from (\ref{Kest}) that $\left\|g_1a - va\right\| \leq \delta\|a\| +
\left\|g_2a\right\| \leq \epsilon$, and hence
that $\left\|g_1a - a\right\| \leq 3\epsilon$. Since $g_1a \in
C^*_r\left({X \backslash Y}\right)$ it follows that $a \in C^*_r\left({X \backslash Y}\right)$.

\end{proof}

Since the Fatou set $F_R$ and the Julia set $J_R$ are totally
$R$-invariant and hence also $\Ro$-invariant we get the following.

\begin{cor}\label{jf-cor} The sequence
\begin{equation*}
\begin{xymatrix}{
0 \ar[r] & C^*_r\left({F_R}\right) \ar[r] & C^*_r\left(R\right) \ar[r]^-{\pi_{J_R}} &C^*_r\left({J_R}\right) \ar[r] & 0.
}\end{xymatrix}
\end{equation*}
is exact.
\end{cor}

\section{The structure of $\J$}\label{JR}

\subsection{Pure infiniteness of $\J$}

\begin{prop}\label{pursimp} $G_{J_R}$ is essentially free and locally
  contractive, and
  $\J$ is purely infinite.
\end{prop}

For the proof of Proposition \ref{pursimp} we need a couple of lemmas.

\begin{lemma}\label{simpleobs} Assume that $\left(R^n\right)'(x) \neq
  0$. Then $R^n(x) \in \Ro(x)$.
\end{lemma}
\begin{proof} When $x$ is not critical for $R^n$ there is an
  open neighbourhood $U$ of $x$ such that $R^n : U \to R^n(U)$ is a local
  transfer and $[x,n,R^n|_U,R^n(x)] \in G_{R}$.
\end{proof}

In the following
proof and in the rest of the paper $\crt$ will denote the set of critical points
of $R$.

%\begin{lemma}\label{7!} The elements of $J_R$ that are neither pre-periodic nor pre-critical are
%  dense in $J_R$.
%\end{lemma}
%\begin{proof} Assume that
%there is a non-empty open subset $U \subseteq J_R$ consisting entirely
%of pre-critical and pre-periodic point. 
%Let $\Per_n R$ denote the set
%of $n$-periodic points of $R$. Since 
%$$
%U\subseteq \bigcup_{n,j} R^{-j}\left( \Per_n R \cup \crt\right)
%$$
%by assumption it follows from the Baire category theorem that there
%are $n, j \in \mathbb N$ and an open non-empty subset $W$ of $J_R$ such that
%$$
%W \subseteq U \cap R^{-j}\left(\Per_n R \cup \crt\right) .
%$$
%Since $R$ is neither periodic nor constant
%it follows that $R^{-j}\left(\Per_n R \cup \crt\right)$
%is finite. Hence $W$ must contain a point $z_0$ which is isolated in
%$J_R$. But there are no isolated points in $J_R$.
%\end{proof} 

\begin{lemma}\label{7!!} Let
  $Y $ be a closed $\Ro$-invariant subset of $\C$. If $Y$ does not
  contain an isolated point which is periodic or critical, it follows that the
  elements of $Y$ that are neither pre-periodic nor pre-critical are
  dense in $Y$.
\end{lemma}
\begin{proof} Assume that
there is a non-empty open subset $U \subseteq Y$ consisting entirely
of pre-critical and pre-periodic point. Let $\Per_n R$ denote the set
of $n$-periodic points of $R$. Since 
$$
U \subseteq \bigcup_{n,j} R^{-j}\left( \Per_n R \cup \crt \right)
$$
by assumption it follows from the Baire category theorem that there
are $n, j \in \mathbb N$ and a non-empty open subset $W$ of $Y$ such that
$$
W \subseteq U  \cap R^{-j}\left(\Per_n R \cup \crt \right) .
$$
Since $R$ is neither periodic nor constant
it follows that $R^{-j}\left(\Per_n R \cup \crt \right)$ is finite. Hence $W$ must contain a point $z_0$ which is isolated in
$Y$. If $z_0 \notin \crt$ we conclude from Lemma \ref{simpleobs} that $R(z_0) \in Y$ since
$Y$ is $\Ro$-invariant. By repeating this argument we either reach an
$l < j$ such that $R^l(z_0) \in \crt  \cap Y$ or conclude that
$R^j(z_0) \in Y$. In the first case $R^l$ is a local transfer in an
open neighbourhood of $z_0$ and in the second $R^j$ is. Hence
$R^l(z_0)$ is isolated in $Y$ in the first case, and $R^j(z_0)$ in the second. In any case we conclude that $Y$
contains an isolated point which is either critical or periodic. This contradicts our
assumption on $Y$.
\end{proof}

We can then give the proof of Proposition \ref{pursimp}:

\begin{proof} By Proposition 4.4 of \cite{Th2} the elements of $J_R$ with
  trivial isotropy group in $G_{J_R}$ are the points that are neither
  pre-periodic nor pre-critical. It is well-known that $J_R$ is
  closed, totally $R$-invariant and without isolated points. It follows therefore from Lemma
  \ref{7!!} that $G_{J_R}$ is essentially free.   

To prove that $G_{J_R}$ is also locally contracting we use that the repelling periodic points are
dense in $J_R$ by ii) of Theorem 14.1 in \cite{Mi}. The argument is
then essentially the
same used in the proof of Lemma 4.2 in \cite{Th3}: Let $U \subseteq
J_R$ be an open non-empty set. There is then a repelling periodic
point $z_0 \in U \cap \mathbb C$, and there is an $n \in
\mathbb N$, a positive number $\kappa > 1$ and an open neighbourhood $W
\subseteq U \cap \mathbb C$ of $z_0$ such that $R^n(z_0) =
z_0$, $R^n$ is injective on $W $
and
\begin{equation}\label{eu1}
\left|R^n(y) -z_0\right| \geq \kappa |y -z_0|
\end{equation}  
for all $y \in W$. Let $\delta_0 > 0$ be so
small that 
\begin{equation}\label{eu2}
 \left\{y \in \mathbb C : |y - z_0| \leq \delta_0 \right\} \subseteq
R^n(W) \cap W .
\end{equation} 
Since $z_0$ is not isolated in $J_R$ there is an element $z_1\in J_R
\cap \mathbb C$ such that $0 < \left|z_1-z_0\right| < \delta_0$. Choose
$\delta$ strictly between $\left| z_1 -z_0\right|$ and $\delta_0$ such that 
\begin{equation}\label{cru}
\kappa
\left|z_1-z_0\right| > \delta.
\end{equation}
Set $V = \left\{y \in J_R \cap \mathbb C: |y -z_0| <  \delta \right\}$. 
Then 
\begin{equation}\label{eu102}
\overline{V}  \subsetneq R^n\left({V} \right). 
\end{equation}
Indeed, if
$\left|y-z_0\right| \leq \delta$ (\ref{eu2}) implies that there is a $y' \in W$ such that
$R^n(y') = y$ and then (\ref{eu1}) implies that $\left|y' - z_0\right|
< \delta$. Since $y' \in J_R$ because $R^{-1}\left(J_R\right) = J_R$, it follows that $\overline{V} \subseteq R^n\left({V} \right)$. On the other hand, it follows from (\ref{cru}) and (\ref{eu1}) that $R^n(z_1) \notin \overline{V}$. This shows
that (\ref{eu102}) holds. Then
$$
S = \left\{ \left[z,n,R^n|_{V},R^n(z)\right] \in G_{J_R} : \ z \in
  V \right\}
$$
is an open bisection in $G_{J_R}$ such that $\overline{V}  \subseteq s(S)$ and 
$$
\alpha_{S^{-1}} \left(\overline{V} \right) \subsetneq {V} ,
$$
where $\alpha_{S^{-1}} : s(S) \to r(S)$ is the homeomorphism defined
by $S$, cf. Section \ref{pure13}. This
shows that $G_{J_R}$ is locally contracting and it follows then from
Theorem \ref{an} that $\J$ is purely infinite.
\end{proof}

\subsection{Exposed points and finite quotients of $\J$}\label{stabjr}

\begin{lemma}\label{escape} Let $X$ be a totally $R$-invariant set which is locally compact in the
topology inherited from $\C$. Let $Y \subseteq X$ be a closed
  $\Ro$-invariant subset of $X$ and let $Y_0$ be the subset of $Y$ obtained by
  deleting the isolated points of $Y$. Then $Y_0$ is totally
  $R$-invariant, i.e $R^{-1}\left(Y_0\right) = Y_0$.
\end{lemma}
\begin{proof} Let $y \in Y_0$. Let $n,m \in \mathbb N$ and consider a
  point $x \in X$ a such that
  $R^n(x) = R^m(y)$. We must show that $x \in Y_0$. Since
  $R^{-n}\left(R^n(x)\right)$ is a finite set there is an open
  neighborhood $U_0$ of $x$ such that 
\begin{equation}\label{escape11}
R^{-n}(R^m(y)) \cap \overline{U_0} = \{x\}.
\end{equation}
Because $R^n$ and $R^m$ are both open maps we can also arange that
there is an open neighborhood $V_0$ of $y$ such that $R^n(U_0) =
R^m(V_0)$. Set $U = U_0 \cap X$ and $V = V_0 \cap X$ and note that
$R^n(U) = R^m(V)$. Since $\bigcup_{j=0}^m R^{-j}(\crt)$ and
  $R^{-m}\left(\bigcup_{j=0}^n R^{j}(\crt)\right)$ are both finite
  sets and $y$ is not isolated in $Y$, there is a sequence $\{y_k\}$
  of mutually distinct elements
  in 
$$
Y\cap V \backslash \left(  R^{-m}\left(\cup_{j=0}^nR^j\left(\crt \right)\right) \cup \bigcup_{j=0}^m R^{-j}(\crt)\right)
$$ 
converging to
  $y$. Since $\# R^{-m}(z)  \leq (\deg R)^m$ for all $z$ we can prune the sequence
  $\{y_k\}$ to arrange that $k \neq l \Rightarrow
  R^m(y_k) \neq R^m(y_l)$. Let $\{x_k\} \subseteq U$ be points such that $R^n(x_k) =
  R^m(y_k)$ for all $k$. Then $x_k \in U \backslash
  \bigcup_{j=0}^nR^{-j}(\crt)$ for all $k$. Passing to a
  subsequence we can arrange that $\{x_k\}$
  converges in $X$, necessarily to $x$ because of
  (\ref{escape11}). Note that $\val (R^m,y_k) = \val
  (R^n,x_k) = 1$ for all $k$ since $x_k \notin \bigcup_{j=0}^nR^{-j}(\crt)$ and
  $y_k \notin \bigcup_{j=0}^mR^{-j}(\crt)$. It follows
  therefore, either from Proposition 4.1 in \cite{Th2} or from Lemma
  \ref{simpleobs} above, that $x_k \in \Ro(y_k)$. Hence $x_k \in Y$ because $y_k \in Y$ and $Y$ is $\Ro$-invariant. It follows that $x \in Y$
  because $Y$ is closed. Furthermore, since the $x_k$'s are distinct,
  $x \in Y_0$. This shows that $Y_0$ is totally $R$-invariant. 
\end{proof}

\begin{lemma}\label{snit} Let $L$ be a closed $\Ro$-invariant
  subset of $J_R$. Then $L$ is either finite or equal to $J_R$.
\end{lemma}
\begin{proof} It follows from Lemma \ref{escape} that we can write $L = L_0 \cup L_1$ where $L_0$ is closed and totally
  $R$-invariant while $L_1$ is discrete. If $L_0 \neq \emptyset$ it
  follows from Corollary 4.13 of \cite{Mi} that $L_0 = J_R =L$. If
  $L_0 = \emptyset$ the compactness of $L$ implies it is finite. 
\end{proof} 

\subsubsection{Exposed points}

The preceding lemma forces us to look for points in $J_R$, or more
generally in $\C$, whose restricted orbits are finite. In the following we say that a point $ x \in  \C$ is \emph{exposed}
when the restricted orbit $\Ro(x)$ of $x$ is finite. A non-empty subset $A \subseteq \C$ is
\emph{exposed} when it is finite and $\Ro$-invariant. Clearly the so-called
exceptional points, those with finite orbit, are exposed. There
are at most two of them and they are always elements of the Fatou
set. See e.g. \S 4.1 in \cite{B}. In contrast exposed points can
occur in the Julia set as well.

\begin{lemma}\label{four0} Let $A \subseteq \C$ be a finite subset
  with the property that 
\begin{equation}\label{four1}
R^{-1}(A) \backslash \crt \subseteq A.
\end{equation}
Then $A$ contains at most 4 elements, and at most 3 if it contains a
critical point.
\end{lemma}
\begin{proof} The proof is a repetition of the proof of Lemma 1 in
  \cite{GPRR}. Set $d = \deg R$ and let $\val(R,x)$ denote the valency
  of $R$ at $x$. Then
\begin{equation*}
\begin{split}
& d \left( \# A\right) = \sum_{x \in R^{-1}(A)} \val (R,x) \ \ \ \ \ \
\ \
\text{(because $\sum_{x \in R^{-1}(y)} \val (R,x) = d$ for all $y$)} \\
& = \# R^{-1}(A) + \sum_{x \in R^{-1}(A)} (\val (R,x) -1) \\
& \leq \# A + \# \crt +   \sum_{x \in R^{-1}(A)} (\val (R,x)
-1) \ \      \ \ \ \ \ \ \ \ \ \ \ \ \text{(by (\ref{four1}))}\\
& = \# A + \# \crt + 2d-2 \ \ \ \ \ \ \ \ \ \ \ \text{(by Theorem 2.7.1 of
  \cite{B})} \\
& \leq \# A + 4(d-1) \ \ \ \ \ \ \ \ \  \text{(by Corollary 2.7.2 of \cite{B}).}
\end{split}
\end{equation*}
It follows that $\# A \leq 4$. When $A$ contains a critical point the
first inequality above is strict and hence $\# A \leq 3$. 
\end{proof}

It follows from Lemma \ref{simpleobs} that an exposed subset satisfies
(\ref{four1}). This gives us the following.

\begin{cor}\label{four} An exposed subset does not contain more than 4 elements. If it contains a
critical point it contains at most 3 elements.
\end{cor}

The upper bound on the number of elements in an exposed set can be
improved when $R$ is a polynomial since the number of critical points
for a polynomial is at most $\deg R$, and $\infty$ is always one
of them. Specifically, applied to a polynomial the proof of Lemma
\ref{four0} yields the following.

\begin{lemma}\label{poly} Assume that $R$ is a polynomial of degree at least 2 and $A$ an exposed subset of $\mathbb C$. Then $\# A \leq 2$,
  and $\# A \leq 1$ when $A$ contains a critical point.
\end{lemma}

We say that an exposed subset is of \emph{type 1} when it does not contain a
critical point, of \emph{type 2} when it contains a pre-periodic
critical point and \emph{type 3} when it contains a critical point,
but no pre-periodic critical point.

\begin{lemma}\label{typ1} Let $A$ be a finite subset of ${\C}$. Then $A$ is an exposed subset of type 1 if and only if 
\begin{equation}\label{maksmir2}
R^{-1}(A) \backslash \crt = A.
\end{equation}
\end{lemma}
\begin{proof} Assume first that $A$ is an exposed subset
  of type 1. Since $A$ does not contain a critical point it follows
  from Lemma \ref{simpleobs} 
  that $R(A) \subseteq A$, i.e. $A \subseteq R^{-1}(A) \backslash
  \crt$. Combined with (\ref{four1}) this shows that
  (\ref{maksmir2}) holds. 

Conversely, it follows from (\ref{maksmir2}) that  
$R(A)
  \subseteq A$ and that $A\cap \crt =
  \emptyset$. Hence $\val (R^n,x) = 1$ for all $x \in A$ and all
  $n \in \mathbb N$. Thus if
  $x \in A$ and $y \in \Ro(x)$ we have both that $R^m(y) =
  R^n(x) \in A$ for some $n,m$, and that $\val (R^m,y) = \val(R^n,x) =1$. But then $R^j(y) \notin \crt$ for
  all $0 \leq j \leq m-1$, and repeated use of (\ref{maksmir2}) then shows
  that $y \in A$, i.e. $A$ is $\Ro$-invariant.
\end{proof}

Finite subsets of the Julia set satisfying (\ref{maksmir2}) were
considered by Makarov and Smirnov in \cite{MS1} and their work can be
used to find examples of polynomials with exposed subsets of type 1 in the Julia
set. In \cite{MS2} a rational map with an exposed subset of type 1 was
called exceptional. This notion was extended in \cite{GPRR} where a
rational map was called exceptional when the Julia set contains a finite
forward invariant subset satisfying (\ref{four1}). There are exceptional rational maps, in the
sense of \cite{GPRR}, without any exposed
subsets, and conversely, there are non-exceptional rational
maps with exposed subsets in the Julia set. Thus although there is of
course a relationship between exposed subsets and the subsets used to
define the exceptional maps in \cite{GPRR}, the two notions are not
the same. Note that it follows from Corollary \ref{four} that the total number of
exposed points is at most 4.

\subsubsection{Finite quotients of $\J$}

Let $E_R$ denote the union of the exposed subsets in $\C$; a set with
at most 4 elements. When $E_R \cap J_R
\neq \emptyset$ the purely infinite $C^*$-algebra $\J$ will have 
$C^*_r\left( E_R \cap J_R\right)$ as a quotient. The corresponding
ideal, however, is simple.

\begin{prop}\label{JRsimple} $C^*_r\left(J_R \backslash E_R \right)$ is
    simple.
\end{prop}
\begin{proof}
It follows from Lemma \ref{7!!}
  that there are many points in $J_R \backslash E_R$ whose isotropy
  group in $G_{J_R \backslash E_R}$ is trivial. By Corollary 2.18 of
  \cite{Th1} it suffices therefore to show that there are no
  non-trivial (relatively) closed $\Ro$-invariant subset in $J_R \backslash
  E_R$. Assume that $L$ is a non-empty $\Ro$-invariant subset of $J_R
  \backslash E_R$ which is closed in $J_R \backslash E_R$. It follows
  first that
  $L \cup \left( E_R \cap J_R\right)$ is closed and $\Ro$-invariant in $J_R$ and then
  from Lemma \ref{snit} that $L = J_R \backslash E_R$, which is the
  desired conclusion. 
\end{proof}

In order to obtain a description of the quotient 
$$
C_r^*(J_R \cap E_R) \simeq \J/C^*_r\left(J_R \backslash E_R\right)
$$
we
consider a more general situation because the result can then be used
when we examine $\F$.

\begin{lemma}\label{discretequot} Let $x \in \C$ and assume that the
  restricted orbit $\Ro(x)$ of $x$ is discrete in
  $\C$. There is an isomorphism
$$
C^*_r\left({\Ro(x)}\right) \simeq C^*\left( \Is_x \right)
\otimes \mathbb K\left( l^2(\Ro(x))\right),
$$
where $\mathbb K\left( l^2(\Ro(x))\right)$ denotes the $C^*$-algebra
of compact operators on $l^2\left(\Ro(x)\right)$ and $\Is_x$ the
isotropy group
$\Is_x = \left\{ g \in G_R : \ s(g) = r(g) = x\right\}$. 
\end{lemma}
\begin{proof} Note that the reduction $G_{\Ro(x)}$ is discrete in $G_R$ and that $C^*_r\left({\Ro(x)}\right)$ is generated by $1_z, \ z \in G_{\Ro(x)}$,
  when $1_z$ denotes the characteristic function of the set
  $\{z\}$. For every $y \in \Ro(x)$ choose an element $\eta_y \in r^{-1}(y)
  \cap s^{-1}(x)$. For every $g \in \Is_x$ set
$$
u_g = \sum_{y \in \Ro(x)} 1_{\eta_y g \eta_y^{-1}} .
$$
The sum converges in the strict topology and defines a unitary in the
multiplier algebra of $C^*_r\left({\Ro(x)}\right)$. Note that $u_gu_h = u_{gh}$ so that $u$ is a unitary
representation of $\Is_x$ as multipliers of
$C^*_r\left({\Ro(x)}\right)$. The elements $1_{\eta_u\eta_v^{-1}}, u,v
\in \Ro(x)$, constitute a system of matrix units generating a copy of
$\mathbb K\left( l^2\left(\Ro(x)\right)\right)$ inside
$C^*_r\left({\Ro(x)}\right)$. Since each $u_g$ commutes with
$1_{\eta_u\eta_v^{-1}}$ for all $u,v$ we get a $*$-homomorphism
$C^*\left(\Is_x\right) \otimes \mathbb K\left( l^2\left(\Ro(x)\right)\right)
\to C^*_r\left({\Ro(x)}\right)$ sending $1_g \otimes 1_{\eta_u
  \eta_v^{-1}}$ to $u_g 1_{\eta_u
  \eta_v^{-1}}$. It is easy to see that this is an isomorphism.
\end{proof}

\begin{lemma}\label{expoinJ} Let $A$ be a finite $\Ro$-orbit in
  $J_R$. Set $ n = \# A$. 
\begin{enumerate}
\item[a)] Assume that $A$ is of type 1. Then $n \leq 4$ and
$$
C^*_r\left(A\right) \simeq C(\mathbb T)
  \otimes M_n(\mathbb C).
$$ 
\item[b)] Assume that $A$ is of type 2. Then $n \leq 3$ and
$$
C^*_r\left(A\right) \simeq M_n(\mathbb C)
  \otimes C(\mathbb T) \otimes \mathbb C^d
$$ 
where $d = \lim_{k \to
    \infty} \val (R^k,x)$ for any critical point $x \in A$.
\item[c)] Assume that $A$ is of type 3. Then $n \leq 3$ and 
$$
C^*_r\left(A\right) \simeq M_n(\mathbb C)   \otimes \mathbb C^d
$$ 
where $d = \lim_{k \to
    \infty} \val (R^k,x)$ for any critical point $x \in A$.
\end{enumerate}
\end{lemma}
\begin{proof} a): It follows from  Corollary \ref{four} that $n \leq 4$. It follows from Lemma \ref{typ1} that $A =\Ro(x)$ for some
  point $x \in J_R$ which is periodic and whose forward orbit is
  contained in $A$. Since $A$ contains no critical
  points Proposition 4.4 b) in \cite{Th2} tells us that $\Is_x \simeq
  \mathbb Z$. Then the conclusion follows from Lemma
  \ref{discretequot}.

b) When $A$ contains a critical point the last assertion in Corollary \ref{four} says that $\#
A \leq 3$. When $x$ is a pre-periodic critical point in $A$ we can
determine the isotropy group $\Is_x$ from Proposition 4.4 in
\cite{Th2}. Since $J_R$ does not contain any periodic critical orbit
we are in situation d2) of that proposition and obtain therefore the
stated conclusion from Lemma \ref{discretequot}.

c) follows in the same way as b). The only difference is that $\Is_x$
is determined by use of c) in Proposition 4.4 of \cite{Th2}.   
\end{proof}

%Note that the isomorphism class of
%$C^*_r\left(A\right)$ determines the type of $A$.

Given a point $z \in \C$ we call the limit $\lim_{k \to \infty}
\val\left(R^k,z\right)$ occuring in Lemma \ref{expoinJ} \emph{the
  asymptotic valency} of $z$. It can be infinite, but only when $z$ is
pre-periodic to a critical periodic orbit.

\begin{thm}\label{Juliaext} There is an extension
\begin{equation}\label{juliaextension}
\begin{xymatrix}{
0 \ar[r] & C^*_r\left({J_R  \backslash
    E_R}\right) \ar[r] & \J \ar[r]^-{\pi_{E_R}} &
\oplus_A C^*_r\left(A\right) \ar[r] & 0,
}\end{xymatrix}
\end{equation}
where the direct sum $\oplus_A$ is over the (possibly empty) collection
of finite $\Ro$-orbits $A$ in $J_R$. Furthermore, $C^*_r\left({J_R  \backslash
    E_R}\right)$ is separable, purely infinite, nuclear, simple and
satisfies the universal coefficient theorem of Rosenberg and Schochet,
\cite{RS}. If non-zero, the
quotient $\oplus_A C^*_r\left(A\right)$ is isomorphic to a finite
  direct sum of matrix algebras $M_n(\mathbb C)$ with $n \leq 3$ and
  circle algebras $C(\mathbb T)\otimes M_n(\mathbb C)$ with $n \leq
  4$. 
\end{thm}
\begin{proof} The extension (\ref{juliaextension}) is a special case
  of the extension from Lemma \ref{quot}. The direct sum decomposition
  of the quotient follows from Lemma
  \ref{basdecomp} and the description of it from Lemma \ref{expoinJ}. The pure infiniteness of
  the ideal follows from Proposition \ref{pursimp} because pure
  infiniteness is inherited by ideals. It is simple by Proposition
  \ref{JRsimple}. That $C^*_r\left({J_R  \backslash
    E_R}\right)$ is nuclear and satisfies the UCT will be shown in
Section \ref{amUCT} below by making a connection to the work of
Katsura, \cite{Ka}.
\end{proof}

In view of Theorem \ref{Juliaext} it seems appropriate to point out that the Julia set can contain
exposed orbits of all three types. For an example of type 1 observe that for the Chebyshev polynomial
$R(z) = z^2 -2$ the set $E_R \cap J_R$
consists of the points $\left\{-2,2\right\}$ which is a finite $\Ro$-orbit of type 1 in the Julia set $[-2,2]$. Hence
$$
\oplus_A C^*_r\left(A\right) \simeq C(\mathbb T)
\otimes M_2(\mathbb C).
$$ 
No other polynomial in the quadratic family $z^2+c$ has an exposed
point in
the Julia set.

For an example of a finite $\Ro$-orbit of type 2 consider the rational
map 
$$
R(z) = \frac{(z-2)^2}{z^2}.
$$ 
The Julia set is the entire sphere
in this case, cf. \S 11.9 in \cite{B}, and the maximal exposed subset consists of the points
$\left\{0,\infty,1\right\}$ which is the union of the finite $\Ro$-orbit
$\{1,\infty\}$ of type 1 and the finite $\Ro$-orbit $\{0\}$ which is of
type 2. The asymptotic valency $\lim_{n \to \infty}
\val\left(R^n,0\right)$ is $2$ and hence  
$$
\oplus_A C^*_r\left(A\right) \simeq C(\mathbb T)
\oplus C(\mathbb T)  \oplus \left(C(\mathbb T)
\otimes M_2(\mathbb C)\right). 
$$

To give examples of finite $\Ro$-orbits of type 3 in the Julia set we
use the work of M. Rees. She shows in Theorem 2 of \cite{R} that for 'many' $\lambda
\in \mathbb C\backslash \{0\}$ the rational map
$$
R(z) = \lambda\left(1 - \frac{2}{z}\right)^2
$$
has a dense critical forward orbit. In particular, the Julia set $J_R$ is the
entire sphere. The critical points are $0$ and $2$, and
$R^{-1}(0) = \{2\}$. Since the forward orbit of $0$ is dense it follows that $\{0\}$ is a finite $\Ro$-orbit of
type 2. There are no
other exposed points, i.e. $E_R  =
\{0\}$. Hence
$$
\oplus_A C^*_r\left(A\right) \simeq \mathbb C \oplus \mathbb C 
$$
in this case because the asymptotic valency of $0$ is $2$.

\subsection{Amenability and the UCT}\label{amUCT}

Set 
$$
J'_R = J_R  \backslash \left( E_R \cup \bigcup_{j=0}^3
  R^{-j}(\crt)\right)
$$ 
and consider  
$$
\Gamma = \left\{ \left[x,k,\eta, y\right] \in G_{J'_R} :
  \ k = 1 \right\} 
$$
which is an open subset of $G_{J'_R}$. Let $X_{\Gamma}$
be the closure of $C_c(\Gamma)$ in $C^*_r\left(J'_R\right)$. Since 
$X_{\Gamma}^* X_{\Gamma} \subseteq C^*_r\left(G^0_{J'_R}\right)$ 
we can consider $X_{\Gamma}$ as a Hilbert $C^*_r\left(G^0_{J'_R }\right)$-module with the 'inner product' $\left< a,b\right> =
a^*b$. Furthermore, since 
$$
C^*_r\left(G^0_{J'_R}\right)X_{\Gamma} \subseteq X_{\Gamma}
$$ 
we can
consider any $a \in C^*_r\left(G^0_{J'_R}\right)$ as an adjointable operator $\varphi(a)$ on
$X_{\Gamma}$. Then the pair $(\varphi, X_{\Gamma})$ is
a $C^*$-correspondence in the sense of Katsura, \cite{Ka}, and we aim
to show that the $C^*$-algebra $\mathcal O_{X_{\Gamma}}$ introduced in
\cite{Ka} is a copy of $C^*_r(J'_R)$.

\begin{lemma}\label{injphi} $\varphi$ is injective. 
\end{lemma}
\begin{proof} Assume that $\varphi(a) = 0$. To show that $a = 0$
  consider the continuous function $j(a)$ on $G^0_{J'_R}$
  defined by $a$, cf. Proposition 4.2 on page 99 in
  \cite{Re}. Consider an element $\gamma \in G_{J'_R}$ such that
  $s(\gamma) \notin R^{-4}(\crt)$. It follows from Lemma \ref{simpleobs} that
  there is a function $f \in C_c(\Gamma)$ supported in $\left\{
    \left[z,1, R|_U, R(z) \right] : z \in U\right\}$ for some open
  neighborhood $U$ of $s(\gamma)$ such that $ff^* \in C_c(J'_R)$ and
  $ff^*(s(\gamma)) = 1$. Since $aff^* = 0$ it follows that
  $j(a)(\gamma) = 0$, i.e. $j(a) = 0$ on
  the set
$$
\left\{ \gamma \in G^0_{J'_R} : \ s(\gamma) \notin R^{-4}(\crt) \right\}
.
$$
Since this set is dense in $G_{J'_R}$ it follows first that
  $j(a) = 0$, and then that $a = 0$ because $j$ is injective.
\end{proof}

Note that we can consider the inclusions $C^*_r\left(G^0_{J'_R}\right) \subseteq C^*_r(J'_R)$ and $X_{\Gamma}
\subseteq C^*_r(J'_R)$ as an injective representation of
the $C^*$-correspondence $(\varphi, X_{\Gamma})$. It follows that there
is an injective $*$-homomorphism $\psi_t : \mathcal K(X_{\Gamma}) \to
C^*_r(J'_R)$ such that
$$
\psi_t\left(\Theta_{a,b}\right) = ab^* ,
$$
cf. Definition 2.3 and Lemma 2.4 in \cite{Ka}. Note that the range of
$\psi_t$ is the ideal 
$$
X_{\Gamma}X_{\Gamma}^* = \overline{ \Span \left\{ ab^* : \ a,b \in
    X_{\Gamma} \right\}} 
$$
in $C^*_r\left(G^0_{J'_R}\right)$. We are here and in the following
lemma borrowing notation from \cite{Ka}.

\begin{lemma}\label{katyes} The ideal
$\left\{ a \in C^*_r\left(G^0_{J'_R}\right) : \ \varphi(a) \in \mathcal K(X_{\Gamma}) \right\}$
is equal to $X_{\Gamma}X_{\Gamma}^*$, and $a = \psi_t(\varphi(a))$ for
all $a \in J_{X_{\Gamma}}$.
\end{lemma}
\begin{proof} The inclusion 
$$
X_{\Gamma}X_{\Gamma}^* \subseteq \left\{ a \in
  C^*_r\left(G^0_{J'_R}\right) : \ \varphi(a) \in \mathcal
  K(X_{\Gamma}) \right\}
$$
is trivial so we focus on the inverse. Let therefore $a \in
C^*_r\left(G^0_{J'_R}\right)$ be an element such that $\varphi(a) \in
\mathcal K(X_{\Gamma})$. There is then a sequence
$$
\sum_{i=1}^{N_n} \Theta_{a^n_i,b^n_i}, \ n = 1,2,3, \dots,
$$
where $a_i^n,b_i^n \in X_{\Gamma}$ for all $i,n$, which converges to
$\varphi(a)$ in
$\mathcal K(X_{\Gamma})$. In particular, 
\begin{equation}\label{katno}
af = \lim_{n \to \infty} \sum_{i=1}^{N_n} \Theta_{a^n_i,b^n_i}f = \lim_{n \to \infty}
\sum_{i=1}^{N_n} a^n_i{b^n_i}^* f
\end{equation}
for all $f \in X_{\Gamma}$. By continuity of $\psi_t$ it follows that the sequence $\psi_t\left( \sum_{i=1}^{N_n}
  \Theta_{a^n_i,b^n_i}\right) = \sum_{i=1}^{N_n}
a^n_i{b^n_i}^*$ converges in $X_{\Gamma}X_{\Gamma}^*$ to $\psi_t(\varphi(a))$. It follows from (\ref{katno}) that $af = \psi_t(\varphi(a))f$
for all $f \in X_{\Gamma}$ and as in the proof of Lemma \ref{injphi} we
deduce first that $j\left(a-\psi_t(\varphi(a))\right) = 0$ and then that $a =\psi_t(\varphi(a))$.   
\end{proof}

It follows from Lemma \ref{katyes} that the
representation of the $C^*$-correspondence $X_{\Gamma}$ given by the inclusions $C^*_r\left(G^0_{J'_R}\right) \subseteq C^*_r(J'_R)$ and $X_{\Gamma}
\subseteq C^*_r(J'_R)$ is covariant in the sense of
Katsura, \cite{Ka}. Combined with the presence of the gauge action on $C^*_r\left(G^0_{J'_R}\right)$
this allows us now to use Theorem 6.4 from \cite{Ka} to conclude that
the $C^*$-algebra $\mathcal O_{X_{\Gamma}}$ defined from the
$C^*$-correspondence $X_{\Gamma}$ is isomorphic to the $C^*$-subalgebra
of $C^*_r\left(J'_R\right)$ generated by
$C^*_r\left(G^0_{J'_R}\right)$ and $X_{\Gamma}$. It remains to show
that that this is all of $C^*_r\left({J'_R}\right)$.

\begin{lemma}\label{crt} For all $x \in J'_R$ there is an element $\gamma \in \Gamma$ such that
  $s(\gamma) = x$.
\end{lemma} 
\begin{proof} For any $z \in J_R \backslash E_R$ and $n \geq 1$, set
$$
D_n(z) = \left\{ y \in R^{-n}(z) : \
  y, R(y), \dots, R^{n-1}(y) \notin \crt \right\}.
$$ 
Let $x \in J'_R$. If
$R^{-1}(x) \nsubseteq \crt$, choose $y \in R^{-1}(x) \backslash
\crt$. It follows then from Lemma \ref{simpleobs} that
$\left[y,1,\eta,x\right] \in \Gamma$ for some local transfer
$\eta$. Assume therefore that $R^{-1}(x) \subseteq \crt$. Then $R(x)
\neq x$ since otherwise $\Ro(x) = \{x\}$, contradicting that $x \notin
E_R$. If $D_2\left(R(x)\right) \neq \emptyset$, choose an element $y
\in D_2\left(R(x)\right)$ and note that $[y,1,\eta,x] \in \Gamma$
for some local transfer $\eta$. Assume therefore now that 
$D_2\left(R(x)\right) = R^{-1}(x) \backslash \crt = \emptyset$. Then $R^2(x) \notin
\left\{x,R(x)\right\}$ since otherwise $\Ro(x) \subseteq \{x,R(x)\}
\cup D_1(R(x))$, contradicting that $x \notin E_R$. If
$D_3\left(R^2(x)\right) \neq \emptyset$, choose $y \in
D_3\left(R^2(x)\right)$ and note that $[y,1,\eta,x] \in \Gamma$
for some local transfer $\eta$. Assume therefore that $D_3\left(R^2(x)\right)
= D_2\left(R(x)\right) = R^{-1}(x) \backslash \crt =\emptyset$. Then
$R^3(x) \notin \{x,R(x), R^2(x)\}$ since otherwise
$\Ro(x) \subseteq \{x,R(x), R^2(x)\}
\cup D_1(R(x)) \cup D_1\left(R^2(x)\right) \cup
D_2\left(R^2(x)\right)$, contradicting that $x \notin E_R$. We claim
that $D_4\left(R^3(x)\right)$ can not be empty. Indeed, if it is
empty we have
either that $R^4(x) \in \left\{x,R(x),R^2(x), R^3(x)\right\}$, and then
\begin{equation*}
\begin{split}
\Ro(x) \subseteq & \{x,R(x), R^2(x), R^3(x)\}
\cup D_1(R(x)) \cup D_1\left(R^2(x)\right) \\
& \ \ \ \ \ \ \ \ \ \ \ \ \ \ \cup
D_2\left(R^2(x)\right) \cup D_1\left(R^3(x)\right) \cup
D_2\left(R^3(x)\right) \cup D_3\left(R^3(x)\right), 
\end{split}
\end{equation*}
which is impossible since $x \in E_R$, or that $R^4(x) \notin
\left\{x,R(x),R^2(x), R^3(x)\right\}$ in which case
\begin{equation*}
\begin{split}
A = & \{x,R(x), R^2(x), R^3(x), R^4(x)\}
\cup D_1(R(x)) \cup D_1\left(R^2(x)\right) \\
& \ \ \ \ \ \ \ \ \ \ \ \ \ \ \cup
D_2\left(R^2(x)\right) \cup D_1\left(R^3(x)\right) \cup
D_2\left(R^3(x)\right) \cup D_3\left(R^3(x)\right), 
\end{split}
\end{equation*}
is a set with more than 4 elements for which (\ref{four1}) holds,
contradicting Lemma \ref{four0}. Thus $D_4\left(R^3(x)\right)$ is not
empty and we choose $y \in
D_4\left(R^3(x)\right)$ and note that $[y,1,\eta,x] \in \Gamma$
for some local transfer $\eta$. 
 \end{proof}

\begin{prop}\label{prop13} $\mathcal O_{X_{\Gamma}} \simeq
  C^*_r\left(J'_R\right)$, 
and $C^*_r\left(J_R \backslash E_R\right)$ is stably isomorphic to $\mathcal O_{X_{\Gamma}}$.
\end{prop}
\begin{proof} By the remarks preceding Lemma \ref{crt}, to establish
  the isomorphism $\mathcal O_{X_{\Gamma}} \simeq
  C^*_r\left(J'_R\right)$ it suffices to show that $C^*_r\left(J'_R\right)$ is generated by $C^*_r\left(G^0_{J'_R}\right)$ and $X_{\Gamma}$. Let $k \geq 0$. Consider a bisection in $G_{J'_R}$ of the
  form
$$
S = \left\{\left[ z,k, \eta, \eta(z)\right] : \ z \in U \right\}
$$
for some local transfer $\eta$ and an open subset $U \subseteq J'_R$ in
the domain of $\eta$ such that $\eta(U) \subseteq J'_R$. By varying $k,\eta$ and $U$ functions
$f \in C_c\left(G_{J'_R}\right)$ with support in a set
of the form $S$ generate $C^*_r\left(J'_R\right)$ as a
$C^*$-algebra. It suffices therefore to show that such functions are
elements of the $*$-algebra generated by $C^*_r\left(G^0_{J'_R}\right)$ and $X_{\Gamma}$. We will prove this by induction in
$k$, starting with the observation that
the assertion is trivial when $k = 0$ and $k=1$. Assume that the
assertion is established for $k-1$ and consider $f \in C_c(G_{J'_R})$ supported in $S$. Let $x \in U$. It follows from Lemma \ref{crt}
that we can find $\gamma \in \Gamma$ such that $s(\gamma) = \eta(x)$. This
means that there is a bisection in $\Gamma$ of the form
$\left\{ \left[z,1,\xi, \xi(z)\right] : z \in V \right\}$
such that $\xi(y) = \eta(x)$ for some $y \in V$. Then
\begin{equation*}\label{uilt}
%begin{split}
 \left[z,k,\eta, \eta(z)\right] 
%= \left[z,k,\eta,
% \eta(z)\right]\left[\eta(z),-1,\xi^{-1},\xi^{-1}(\eta(z))\right]\left[\xi^{-1}(\eta(z)\right),%,\xi, \eta(z)] \\
  = \left[ z,k-1,\xi^{-1} \circ \eta,
  \xi^{-1}(\eta(z))\right]\left[\xi^{-1}(\eta(z)\right),1,\xi,
\eta(z)]
%end{split} 
\end{equation*}
for all $z$ in a neighborhood $U_x \subseteq U$ of $x$. It follows
that there are functions $h,g \in C_c(G_{J'_R})$ supported in 
$$
\left\{ \left[ z,k-1,\xi^{-1} \circ \eta,
  \xi^{-1}(\eta(z))\right] : \ z \in U_x \right\}
$$
and
$$
\left\{ \left[\xi^{-1}(\eta(z)\right),1,\xi,
\eta(z)] : \ z \in U_x \right\},
$$
respectively, such that $f = hg$ in a neighborhood of $\left[x,k,\eta,
  \eta(x)\right]$. Note that $h$ is then in the $*$-algebra generated by $C^*_r\left(G^0_{J'_R}\right)$ and $X_{\Gamma}$ by induction hypothesis and that $g
\in X_{\Gamma}$. We choose then functions $\psi_i \in C_c(J'_R)$ forming a finite partition of
unity $\left\{\psi_i\right\}$ on
$r\left(\supp f\right)$ such that 
$$
f = \sum_i \psi_i h_ig_i
$$
where $g_i \in X_{\Gamma}$ and $h_i$ is in the $*$-algebra generated by $C^*_r\left(G^0_{J'_R}\right)$ and $X_{\Gamma}$. Since $\psi_i \in C^*_r\left(G^0_{J'_R}\right)$ for each $i$, this completes the induction step and
hence the proof of the isomorphism $\mathcal O_{X_{\Gamma}} \simeq
  C^*_r\left(J'_R\right)$.

Since $G_{J'_R}$ is the reduction to the open subset $J'_R$ of $J_R \backslash E_R$ the algebra
$C^*_r(J'_R)$ is a hereditary $C^*$-subalgebra of $C^*_r(J_R
\backslash E_R)$. The latter algebra is simple by Proposition
\ref{JRsimple} and it follows therefore from \cite{Br} that the two algebras are
stably isomorphic.
       
\end{proof}

\begin{cor}\label{UCT..} $C^*_r\left(J_R
    \backslash E_R\right)$ is a nuclear
  $C^*$-algebra which satisfies the universal coefficient theorem of
  Rosenberg and Schochet, \cite{RS}.
\end{cor}
\begin{proof} Thanks to Proposition \ref{prop13} and Lemma \ref{katyes} the assertions follow from
  Corollary 7.4 and Proposition 8.8 in \cite{Ka} provided we show that both
  $C^*_r\left(G^0_{J'_R}\right)$ and the ideal $X_{\Gamma}X_{\Gamma}^*$ are nuclear
and satisfy the UCT. To this end we use the inductive limit
decomposition
\begin{equation}\label{union2}
C^*_r\left(G^0_{J'_R}\right) = \overline{\bigcup_n C^*_r\left(G^0_{J'_R}(n)\right)},
\end{equation}
cf. (\ref{union}). We claim that $C^*_r\left(G^0_{J'_R}(n)\right)$ is liminary for each $n$. In fact we will show
that every irreducible representation  of $C^*_r\left(G^0_{J'_R}(n)\right)$ is finite dimensional for each $n$. By Theorem 6.2.3 in \cite{Pe} this
will show that $C^*_r\left(G^0_{J'_R}(n)\right)$ is a type I $C^*$-algebra and it follows then from
(\ref{union2}) that $C^*_r\left(G^0_{J'_R}\right)$ is nuclear and satisfies the UCT. Furthermore, since
we also have the inductive limit decomposition
\begin{equation}\label{union3}
X_{\Gamma}X_{\Gamma}^* = \overline{\bigcup_n C^*_r\left(G^0_{J'_R}(n) \right) \cap X_{\Gamma}X_{\Gamma}^* },
\end{equation}
and since we know that every irreducible representation of $C^*_r\left(G^0_{J'_R}(n) \right) \cap X_{\Gamma}X_{\Gamma}^*$ is finite dimensional
when this holds for $C^*_r\left(G^0_{J'_R}(n) \right)$, we can at the same time conclude that also
$X_{\Gamma}X_{\Gamma}^*$ is nuclear and satisfies the UCT.

To show that $C^*_r\left(G^0_{J'_R}(n)\right)$ is liminary consider an irreducible representation $\pi$ of $C^*_r\left(G^0_{J'_R }(n)\right)$. Every function $f \in C(\C)$ defines a multiplier
$\psi(f)$ of $C^*_r\left(G^0_{J'_R}(n)\right)$ such that
$$
\psi(f)g[x,0,\eta,y] = f\left(R^n(x)\right)g[x,0,\eta,y] 
$$
when $g \in C_c\left(G^0_{J'_R}(n)\right)$. Note that $\psi(f)$ is central in the multiplier algebra and that
$\psi$ is a $*$-homomorphism. Since $\pi$ is irreducible there is a point $z \in \C$ such that $\pi\left(\psi(f)g\right)  = f(z)\pi(g)$ for all $f \in
C(\C)$ and all $g \in C^*_r\left(G^0_{J'_R}(n)\right)$. Consequently $\pi(g) = 0$ for every $g \in C_c\left(G^0_{J'_R}(n)\right)$ whose support does not contain elements from
$F = r^{-1}\left(R^{-n}(z)\right)$. Since all the isotropy groups in
$G^0_{J'_R}(n)$ are finite by Lemma 4.2 in \cite{Th2} it follows that $F$ is a
finite set. Therefore $\pi\left( C_c\left(G^0_{J'_R}(n)\right)\right)$
must be finite dimensional and the same is then
necessarily true for $\pi$.     
\end{proof} 

Note that the $C^*$-correspondance $X_{\Gamma}$ represents an element
$$
\left[X_{\Gamma}\right] \in KK\left(X_{\Gamma}X_{\Gamma}^*, C^*_r\left(G^0_{J'_R}\right)\right).
$$ 
This element defines a homomorphism 
$$
\left[X_{\Gamma}\right]_* : \ K_*\left(X_{\Gamma}X_{\Gamma}^*\right) \to K_*\left( C^*_r\left(G^0_{J'_R}\right)\right)  
$$
which fits into the following six terms exact sequence, cf. Theorem
8.6 in \cite{Ka}.

\begin{cor}\label{cor2} There is an exact
  sequence
\begin{equation}
\begin{xymatrix}{
K_0\left( X_{\Gamma}X_{\Gamma}^*\right) \ar[rr]^-{\id -
  \left[X_{\Gamma}\right]_*} & & K_0\left( C^*_r\left(G^0_{J'_R}\right)\right) \ar[rr]^{\iota_*}  & & K_0\left( C^*_r\left({J_R \backslash
    E_R}\right)\right)\ar[d] \\
K_1\left( C^*_r\left({J_R \backslash
    E_R}\right)\right)\ar[u]  & & K_1\left( C^*_r\left(G^0_{J'_R}\right)\right) \ar[ll]^{\iota_*}  & &  K_1\left( X_{\Gamma}X_{\Gamma}^*\right) \ar[ll]^-{\id -
  \left[X_{\Gamma}\right]_*}
}\end{xymatrix}
\end{equation} 
where $\iota : C^*_r\left(G^0_{J'_R}\right) \to C^*_r\left({J_R \backslash
    E_R}\right)$ is the inclusion $C^*_r\left(G^0_{J'_R}\right) \subseteq  C^*_r\left(J'_R\right)$ followed by the
  stable isomorphism $C^*_r\left({J'_R}\right)  \simeq C^*_r\left({J_R \backslash
    E_R}\right)$.
\end{cor}

It may seem possible to calculate the $K$-theory of $C^*_r\left({J_R \backslash
    E_R}\right)$ from Corollary \ref{cor2} and the inductive limit decompositions
(\ref{union2}) and (\ref{union3}). In practice, however, the task is
very difficult, not only because the six terms exact sequence of
Corollary \ref{cor2} is less helpful than the one which is available
for local homeomorphisms, \cite{DM}, \cite{Th3}, and which can be applied
here when there are no critical points in the Julia set, but also
because the topology of $J_R$ is poorly understood in general.

\section{The structure of $\F$}

It is well-known and not difficult to see that $R$ takes a connected
component $W$ of $F_R$ onto another connected component $R(W)$ of
$F_R$. It follows that we can define an equivalence relation $\sim$
in $F_R$ such that $x \sim y$ if and only if there are $n,m \in
\mathbb N$ such that $R^n(x)$ and $R^m(y)$ are contained in
the same connected component of $F_R$. By Sullivan's
no-wandering-domain theorem, Theorem 16.4 in \cite{Mi}, and a result
of Shishikura, \cite{Sh}, the set of equivalence classes $F_R/\sim$ is
finite, and in fact can not have more that $2\deg R-2$ elements. We can therefore write
\begin{equation*}\label{sullivandecompmmmmm}
F_R = \sqcup_{i = 1}^N \Omega_i,
\end{equation*}
where $N \leq 2 \deg R - 2$, such that each $\Omega_i$ is open, $R^{-1}\left(\Omega_i\right) =
 \Omega_i$ and $\Omega_i \cap \Omega_j =
 \emptyset$ when $i \neq j$. The $\Omega_i$'s will be called the
 \emph{stable regions} of $R$. Since they are $\Ro$-invariant
 it follows from Lemma \ref{basdecomp} that
\begin{equation}\label{sullivandecomp2}
C^*_r\left({F_R}\right) = \oplus_{i=1}^N C^*_r\left( {\Omega_i}
\right) .
\end{equation}

The stable regions are divided into different types reflecting
the faith of their elements under iteration.

\begin{defn}\label{fatoudomains} Let $U$ be an open subset of
  $F_R$ and $p \in \mathbb N$. We say that
\begin{itemize} 
\item[a)] $U$ is a \emph{super-attracting domain} of period $p$ when
  $R^p(U) \subseteq U$, $R^i(U) \cap U = \emptyset, 1\leq i \leq p-1$, and there are
  a natural number $d \geq 2$, an $r \in ]0,1[$ and a conformal conjugacy $\psi : U
  \to D_r = \left\{ z \in \mathbb C : \left|z \right| < r \right\}$
  such that
\begin{equation*}\label{dom1}
\begin{xymatrix}{ U \ar[rr]^-{R^p} \ar[d]_-{\psi} & & U \ar[d]^-{\psi} \\
D_r \ar[rr]_-{z \mapsto z^d} && D_r
}\end{xymatrix} 
\end{equation*}
commutes;
\item[b)] $U$ is an \emph{attracting domain} of period $p$ when
  $R^p(U) \subseteq U$, $R^i(U) \cap U = \emptyset, 1\leq i \leq p-1$, and there are
  a $\lambda \in \mathbb C, \ |\lambda| < 1$, an $r > 0$ and a conformal conjugacy $\psi : U
  \to D_r = \left\{ z \in \mathbb C : \left|z \right| < r \right\}$
  such that 
\begin{equation*}\label{dom2}
\begin{xymatrix}{ U \ar[rr]^-{R^p} \ar[d]_-{\psi} &&U \ar[d]^-{\psi} \\
D_r \ar[rr]_-{z \mapsto \lambda z} && D_r
}\end{xymatrix} 
\end{equation*}
commutes;
\item[c)] $U$ is a \emph{parabolic domain} of period $p$  when
  $R^p(U) \subseteq U$, $R^i(U) \cap U = \emptyset, 1\leq i \leq p-1$, and there is a conformal
  conjugacy $\alpha : U \to \mathbb H = \left\{z \in \mathbb C : \ \Real
    \ z > 0 \right\}$ such that
\begin{equation*}\label{dom3}
\begin{xymatrix}{ U \ar[rr]^-{R^p} \ar[d]_-{\alpha} &&U \ar[d]^-{\alpha} \\
\mathbb H \ar[rr]_-{z \mapsto z+1} && \mathbb H
}\end{xymatrix} 
\end{equation*}
commutes;
\item[d)] $U$ is a \emph{Siegel disk} of period $p$ when
  $R^p(U) = U$, $R^i(U) \cap U = \emptyset, 1\leq i \leq p-1$, and there are a $t \in \mathbb R \backslash \mathbb Q$ and a conformal
  conjugacy $\psi : U \to D_1 = \left\{z \in \mathbb C : \ |z| <
    1\right\}$ such that 
\begin{equation*}\label{dom4}
\begin{xymatrix}{ U \ar[rr]^-{R^p} \ar[d]_-{\psi} &&U \ar[d]^-{\psi} \\
D_1 \ar[rr]_-{z \mapsto e^{2\pi it} z} && D_1
}\end{xymatrix} 
\end{equation*}
commutes;
\item[e)] $U$ is a \emph{Herman ring} of period $p$  when
  $R^p(U) = U$, $R^i(U) \cap U = \emptyset, 1\leq i \leq p-1$, and there are a $t \in \mathbb R \backslash \mathbb Q$ and a conformal
  conjugacy $\psi : U \to \mathbb A = \left\{z \in \mathbb C : 1 <  |z| <
    2\right\}$ such that 
\begin{equation*}\label{dom5}
\begin{xymatrix}{ U \ar[rr]^-{R^p} \ar[d]_-{\psi} && U \ar[d]^-{\psi} \\
\mathbb A \ar[rr]_-{z \mapsto e^{2\pi it} z} && \mathbb A
}\end{xymatrix} 
\end{equation*}
commutes.
\end{itemize}
\end{defn}

It follows from the classification of the periodic Fatou components, \cite{Mi}, that a stable region $\Omega$ contains a domain $U$ of one of the types described in a)-e) of Definition
\ref{fatoudomains} with the property that 
\begin{equation}\label{5}
\Omega = \bigcup_{n=0}^{\infty} R^{-n}(U).
\end{equation} 
We will say that $\Omega$ is a super-attractive, attractive,
parabolic, Siegel or Herman region in accordance with the nature of the
domain $U$ which we will refer to as a \emph{core} of $\Omega$.

\begin{lemma}\label{perclosed} Let $x \in F_R$ be periodic. Then $\Ro(x)$
  is closed and discrete in $F_R$.
\end{lemma}
\begin{proof} Let $W$ be a connected component of $F_R$ and $K
  \subseteq W$ a compact subset. We must show that $K \cap \Ro(x)$ is
  finite. Assume therefore that $W \cap \Ro(x) \neq \emptyset$. Let
  $\Omega$ be the stable region containing $W$ and $U$ a core for
  $\Omega$. It follows from (\ref{5}) and the compactness of $K$ that
  there is an $l \in \mathbb N$ such that $R^l(K) \subseteq U$. When $y
  \in K \cap \Ro(x)$ there is then a $k > l$ such that $R^k(y) = x$. Thus
  $R^l(y)$ is a pre-periodic element of $U$. By inspecting the
  possible types of $U$ we see that $U$ contains at most one point
  pre-periodic under $R$. Thus $R^l(y) = x$, proving that $K \cap \Ro(x)
  \subseteq R^{-l}(x)$ which is a finite set.
\end{proof}

\begin{lemma}\label{equal} $\bigcup_{c \in \crt} \Ro(c) = \bigcup_{n=0}^{\infty} R^{-n}(\crt)$.
\end{lemma}
\begin{proof} Assume that $x \in \bigcup_{c \in \crt} \Ro(c)$. There is then a
  critical point $c \in \crt$ and $n,m \in \mathbb N$ such that
  $R^n(x) = R^m(c)$ and $\val\left(R^n,x\right) =
  \val\left(R^m,c\right)$. Since $\val(R^m,c) \geq 2$ this implies
  that $\val(R^n,x) = \val(R,R^{n-1}(x))
    \val(R,R^{n-2}(x)) \cdots \val(R,x) \geq 2$ and hence that $x \in
    \bigcup_{j=0}^{n-1} R^{-j}(\crt)$. Conversely, if $x \in
    \bigcup_{n=0}^{\infty} R^{-n}(\crt)$ let $j \in \mathbb N$
    be the least natural number such that $R^j(x) \in \crt$. It
    follows from Lemma \ref{simpleobs} that $x \in
    \Ro\left(R^j(x)\right) \subseteq  \bigcup_{c \in \crt} \Ro(c)$.
\end{proof}

\begin{lemma}\label{gcclosed} Let $x$ be a critical point in
  $F_R$. Then $\Ro(x)$ is 
  closed and discrete in $F_R$.
\end{lemma}
\begin{proof} As in the proof of Lemma \ref{perclosed} we take a
  connected component $W$ of $F_R$ and show that $K \cap \Ro(x)$ is finite for
  any compact subset $K \subseteq W$. Let $\Omega$ be the stable
  region containing $W$ and $U$ a core for $\Omega$ of period
  $p$. There is an $l \in \mathbb N$ such that $R^l(K) \subseteq U$. By inspecting the possible types of $U$ we see that there is at
  most one element of $U$ which is pre-critical under $R$, and that
  element is already critical for $R^p$ when it exists. It follows
  therefore from Lemma \ref{equal} that $K \cap \Ro(x) \subseteq
  \bigcup_{n=0}^{l+p} R^{-n}(\crt)$ which is a finite set. 
\end{proof}

 Since there are only
finitely many periodic and critical points in $F_R$ the union of their
restricted orbits, which we will denote by $\mathcal I$, is closed and
discrete in $F_R$ by Lemma \ref{perclosed} and Lemma
\ref{gcclosed}. It follows therefore from Lemma \ref{quot} that for
each stable region $\Omega$ in $F_R$ there is an extension
\begin{equation}\label{sullivanextension}
\begin{xymatrix}{
0 \ar[r] & C^*_r\left({\Omega \backslash \mathcal I}\right) \ar[r]
& C^*_r\left(\Omega \right) \ar[r]^-{\pi_{\mathcal I}} &
 C^*_r\left(\Omega \cap {\mathcal I}\right) \ar[r] & 0.
}\end{xymatrix}
\end{equation}

To study the ideal $C^*_r\left({\Omega \backslash \mathcal
    I}\right)$ we need the following lemma which seems to be folklore
for mathematicians working with rational maps. We sketch a proof
for the benefit of the operator algebraists.

\begin{lemma}\label{3} Let $U \subseteq \overline{\mathbb C}$ be an open simply
  connected subset such that 
$$
U \cap \bigcup_{n=1}^{\infty} R^{n}\left(
    \crt \right) = \emptyset.
$$ 
Let $d = \deg R$ be the degree of $R$. For each $n \in \mathbb N$ there are $d^n$ open connected subsets
$W^n_1,W^n_2, \dots , W^n_{d^n}$ and $d^n$ holomorphic maps $\chi^n_i : U \to
W^n_i, i = 1,2, \dots, d^n$, such that
\begin{enumerate}
\item[i)] $R^n\left(\chi^n_i(z)\right) = z, \ z \in U$,
\item[ii)] $\chi^n_i(U) = W^n_i, i  = 1,2, \dots, d^n$,
\item[iii)] $W^n_i \cap W^n_j = \emptyset, \ i \neq j$, and
\item[iv)] $R^{-n}(U) = \bigcup_{i=1}^{d^n} W^n_i$. 
\end{enumerate}
\end{lemma}
\begin{proof} Note that $R^n$ is a $d^n$-fold covering of $U$ by
  $R^{-n}(U)$ since $U \cap \bigcup_{j=1}^n R^j(\crt) = \emptyset$. Let $W_i, i =1,2, \dots, N$, be the connected components of $R^{-n}(U)$. We
  claim that
  $R^{n}(W_i) = U$. To see this let $x \in W_i$ and let $y \in U$. Since
  $R^n : R^{-n}(U) \to U$ has the path-lifting property we can lift a
  path in $U$ connecting $R^n(x)$ to $y$ in $U$ to a path starting
  in $x$. This path must end in a point in $W_i$ which maps to $y$ under
  $R^n$, proving the claim. Then $R^n : W_i \to U$ is also a covering
  and since $U$ is simply connected it must be a homeomorphism. Let
  $\chi^n_i : U \to W_i$ be its inverse and note that $\chi^n_i$ is
  holomorphic since $R^n$ is. It follows also that $N = d^n$ since $\#
  R^{-n}(y) = d^n$ for all $y \in U$. The proof is complete.

%Next we prove that $R^n : W_i \to U$ is a covering map: Let $y \in U$
%and let $x_1,x_2, \dots, x_j$ be the elements of $R^{-n}(y) \cap
%W_i$. For each $x_k$ there an open neighborhood $\Omega_k \subseteq W_i$ of $x_k$
%such that $R^n$ is a homeomorphism mapping $\Omega_k$ onto the open
%subset $R^n(\Omega_k)$ of $U$. We arrange that the $\Omega_k$'s are mutually disjoint. 
%Assume that there is a sequence $\{y_i\}$
%converging to $y$ in $U$ such that for each $i$ there is an element
%$z_i \in R^{-n}(y_i) \backslash \left( \bigcup_{k=1}^j
%  \Omega_k \cup \bigcup_{i' \neq i} W_{i'}\right)$ such that $R^n(z_i)
%= y_i$. A condensation point of $\{z_i\}$ will then give an element of
%$R^{-n}(y)$ outside of $ \bigcup_{k=1}^j
%  \Omega_k \cup \bigcup_{i' \neq i} W_{i'}$, a contradiction. There is
%  therefore an open neighborhood $U_0$ of $y$ such that $U_0 \subseteq
%  \bigcap_{k=1}^j R^n\left(\Omega_k\right)$ and $R^{-n}(U_0)
%  \cap W_i \subseteq \bigcup_{k=1}^j \Omega_k$. Hence $R^{-n}(U_0)
%  \cap W_i = \bigcup_{k=1}^j \Omega_k \cap R^{-n}(U_0)$ and $R^n :
%  \Omega_k \cap R^{-n}(U_0) \to U_0$ is a homeomorphism for each $k$. -
%  Thus $R^n : W_i \to U$ is a covering map. Since $W_i$ is connected
%  and $U$ simply connected $R^n : W_i \to U$ is a homeomorphism. Let
%  $\chi^n_i : U \to W_i$ be its inverse and note that $\chi^n_i$ is
%  holomorphic since $R^n$ is. It follows also that $N = d^n$ since $\#
%  R^{-n}(y) = d^n$ for all $y \in U$. The proof is complete.
\end{proof}

\begin{lemma}\label{16} Let $\Omega$ be a stable region for $R$ and $U \subseteq \Omega$ a core for $\Omega$. Then
$$
C^*_r\left({\Omega \backslash \mathcal I}\right) \simeq  C^*_r\left({U \backslash \mathcal I}\right) \otimes \mathbb K.
$$ 
\end{lemma}
\begin{proof} Let $x \in \Omega \backslash \mathcal I$. It follows
  from (\ref{5}) that there is a $k \in \mathbb N$ such that $R^k(x)
  \in U$. Since $x \notin \mathcal I$ it follows from Lemma
  \ref{equal} and Lemma \ref{simpleobs} that $R^k(x) \in \Ro(x)$. This
  shows that $\Ro(x) \cap \left(U \backslash \mathcal I\right)
  \neq \emptyset$ and
  it follows then from Theorem \ref{MRW} that $C^*_r\left({\Omega
      \backslash \mathcal I}\right)$ is stably isomorphic to
  $C^*_r\left({U \backslash \mathcal I}\right)$. It suffices
  therefore now to show that $C^*_r\left({\Omega \backslash
      \mathcal I}\right)$ is stable. To this end we use Lemma
  \ref{stability} and consider therefore a compact subset $K$ of
  $\Omega \backslash \mathcal I$. The construction of the required
  bi-sections will be performed
  differently for the different core domains.

Assume first that $U$ is a Siegel disk or a Herman ring. It follows
from Theorem 16.1 in \cite{Mi} that $U$ is a connected component of
$F_R$ in these cases. There is an
$l \in \mathbb N$ such that $R^l(K) \subseteq U$. Note that when
$U$ is a Siegel disk the periodic point at the center of $U$ is not
in $R^l(K)$. For each $z \in \crt \cap \Omega$, let 
$\chi(z)$ be the first element from $U$ in the forward orbit of $z$,
i.e. $\chi(z) \in U$ is the element determined by the condition that
$R^m(z) = \chi(z)$ while $R^i(z)  \notin U, i = 0,1,2, \dots,
m-1$. Then $\chi(\crt)$ is a finite (possibly empty)
set. Since the rotation in the core is irrational and $K \cap
\bigcup_{j=0}^{\infty} R^{-j}(\crt) = \emptyset$, there is for each point $x$ in $K$ an open neighbourhood $V_x$ of
$x$ such that $R^i$ is injective on $V_x$ for all $i \in \mathbb N$
and an $n_x \in \mathbb N$ with the property that
$R^{n_x + l}\left(V_x\right) \cap \chi(\crt) = \emptyset$. We
can also arrange that $R^{n_x + l}(V_x)$ is simply connected (e.g. a
small disc). By compactness of $K$ there is a finite collection $V_{x_i}, i =1,2, \dots, N$, such that $K
\subseteq \bigcup_{i=1}^N V_{x_i}$. Let $U_{-1},U_{-2}, U_{-3}, \dots$
be a sequence of Fatou components such that $U_{-1} \notin
\left\{U,R(U), R^2(U), \dots, R^{p-1}(U)\right\}$ and
$R\left(U_{-i}\right) = U_{-i +1}, i \geq 2$. Such a sequence exists
since $R^{-p}\left(U\right) \nsubseteq U$. Note that $U_{-i} \cap
U_{-j} = \emptyset$ when $i \neq j$. Since $K$ is compact there is an
$m \in \mathbb N$ such that
$$ 
K \cap U_{-j} = \emptyset, \ j \geq m.
$$
Let $i \in
\left\{1,2,\dots, N\right\}$. By using Lemma \ref{3} we choose a connected subset
$V'_i \subseteq U_{-i -m}$ such that $R^{i+m}$ is a homeomorphism from $V'_i$
onto $R^{n_{x_i} +
  l}\left(V_{x_i}\right)$. Let $R^{-i-m} : R^{n_{x_i} +
  l}\left(V_{x_i}\right) \to V'_i$ denote its inverse. Then  
$$
S_i = \left\{\left( R^{-i-m} \circ R^{n_{x_i} + l}(z), i + m -n_{x_i}-l , \left(R^{-i-m}
      \circ R^{n_{x_i} + l}|_{V_{x_i}}\right)^{-1}, z\right) : \ z \in V_i \right\}  ,
$$  
$ \ i =1,2,
\dots, N$, is a collection of open bi-sections in $G_{\Omega \backslash \mathcal I}$
meeting the requirements in Lemma \ref{stability}.

Assume instead that $U$ is attracting or parabolic. We choose first $L$
such that $R^L(K) \subseteq U$, and
then a finite open and relatively compact cover $V_1,V_2,
\dots, V_N$ of $K$ in $\Omega \backslash \mathcal I$ such that $R^L : V_i \to R^L(V_i) \subseteq U$ is injective for
each $i$. Subsequently we choose $n_1,n_2, \dots, n_N \in \mathbb N$ such
that $R^{n_i}\left(R^L(V_i)\right)) \cap R^{n_j}\left(R^L(V_j)\right) =
  \emptyset$ when $i \neq j$ and
$$
R^{n_i + L}(V_i) \cap K = \emptyset
$$
for all $i$. Then
$$
S_i = \left\{ \left(R^{n_i+L}(z), -n_i -L, \left(R^{n_i +L}|_{V_i}\right)^{-1}, z\right) : \ z
    \in V_i \right\} , 
$$
$i = 1,2, \dots, N$, is a collection of open bi-sections in $G_{\Omega \backslash \mathcal I}$
meeting the requirements in Lemma \ref{stability}. 

Finally, in the super-attractive case choose $L$ such that $R^L(K) \subseteq
U$ and $K \cap R^i(K) = \emptyset, i \geq L$. Let $U_0$ be an open
subset of $U$ such that $R^L(K) \subseteq U_0$ and $\overline{U_0}$ is
a compact subset of $U \backslash \{x\}$, where $x$ is the critical
point in $U$. Set 
$$
Y =  R^{-L}\left( \bigcup_{n=1}^{\infty}R^n(\crt) \right) \cap K 
$$
and note that $Y$ is a finite set. For every $z \in Y$ choose a small neighbourhood $V_z$ of $z$ and a natural number
$n_z$ such that $R^{n_z +L}$ is injective on $\overline{V_z}$, $R^{n_z
  +L}\left(\overline{V_z}\right) \subseteq U \backslash \{x\}$ and 
$$
R^{n_z + L}(V_z) \cap K = \emptyset
$$
for all $z$, and $R^{n_z
  +L}\left(\overline{V_z}\right) \cap R^{n_{z'}+L}\left(\overline{V_{z'}}\right) = \emptyset$
when $z \neq z'$. This is possible because $K \cap \mathcal I =
\emptyset$, cf. Lemma \ref{equal}. Then $R^L\left(K\backslash
  \bigcup_{z\in Y} V_z\right) \cap \bigcup_{n=1}^{\infty}R^n(\crt) = \emptyset$. We can therefore cover $K
\backslash \bigcup_{z \in Y} V_z$ by a finite collection $W_i, i = 1,2, \dots,
N$, of open sets such that $R^L(W_i)$ is an open simply
connected subset of $U_0 \backslash \bigcup_{n=1}^{\infty}R^n(\crt)$. It follows then from Lemma
\ref{3} that for any collection $n_i, i = 1,2, \dots, N$, of natural
numbers we can find univalent holomorphic maps $\chi_{i} :
R^L\left(W_i\right) \to R^{-n_i}\left(U_0\right)$ such that $R^{n_i} \circ
  \chi_i(z) = z$ for all $z \in R^L\left(W_i\right)$. Set $U_1 = U_0
  \cup \bigcup_{z \in Y} R^{n_z +L}\left(V_z\right)$ and note that $U_1$ is a
  relatively compact subset of $U \backslash \{x\}$. There is therefore
  an $N_1 \in \mathbb N$ such
  that $R^i(U_1) \cap U_1 = \emptyset, i \geq N_1$. Thus, if we
  arrange that $N_1 + L\leq n_1$ and $n_i \geq n_{i-1} + N_1, i = 2,3,
  \dots, N$, we will have that the sets
$$
\chi_i \circ R^{L}(W_i), \ i = 1,2, \dots, N,
$$
are mutually disjoint, and also disjoint from $K \cup U_1$. For each
$z \in Y$ the set
$$
T_z = \left\{ \left( R^{n_z+L}(z), -n_z -L, \left(R^{n_z + L}|_{V_z}\right)^{-1},
    z\right) : \ z \in V_z \right\}
$$
is an open bi-section in $G_{\Omega \backslash \mathcal I}$. The same is true
  for
$$
S_i = \left\{ \left( \chi_i \circ R^L(z), n_i-L, \left(\chi_i \circ
      R^L|_{W_i}\right)^{-1}, z\right) : \ z \in W_i \right\},
$$
$ i = 1,2, \dots, N$. Taken together we have a collection of
bi-sections in $G_{\Omega \backslash \mathcal I}$ with the properties
required in Lemma \ref{stability}.

\end{proof}

\subsection{Super-attractive stable regions}\label{sssuper-att} 

In this section we study the $C^*$-algebra $C^*_r\left(\Omega
  \backslash \mathcal I\right)$ in the case where $\Omega$ is a
super-attractive stable region. Let $U$ be a core domain for $\Omega$. Then
$\mathcal I \cap U$ consists only of the super-attracting periodic
point $x$ at the center of
$U$. It follows from Lemma \ref{16} and Proposition \ref{nocrit} that $C^*_r\left({\Omega
    \backslash \mathcal I}\right) \simeq C^*_r\left(\psi\right)
\otimes \mathbb K$
where $\psi : D_r \backslash \{0\} \to D_r \backslash \{0\}$, for
some $r \in ]0,1[$, is the
local homeomorphism $\psi(z) = z^d$. Note that $d = \val(R^p,x) \geq 2$ where $p$ is the period of $x$. Let $D = \left\{z \in \mathbb C : \ 0 < |z| <
  1\right\}$. Define $\alpha : D \to D$
such that $\alpha(z) = z^d$. Since $D = \bigcup_j \alpha^{-j}\left(D_r
\backslash \{0\}\right)$ it follows from Theorem \ref{MRW} that
$C^*_r(\psi)$ is stably isomorphic to $C^*_r(\alpha)$. Thus $C^*_r\left({\Omega
    \backslash \mathcal I}\right) \simeq C^*_r\left(\alpha\right)
\otimes \mathbb K$ since $C^*_r\left({\Omega
    \backslash \mathcal I}\right)$ is stable by Lemma \ref{16}, and in
this section we identify the stable isomorphism
class of $C^*_r\left(\alpha\right)$.

First identify $D$ with
$]0,1[ \times \mathbb T$ via that map $(t,\lambda) \mapsto
t\lambda$. In this picture 
$$
\alpha(t,\lambda) = \left(t^d,\lambda^d\right)
.
$$
Map $]0,1[ \times \mathbb T$ to $\mathbb R \times \mathbb T$ using
the map
$$
(t,\lambda) \mapsto \left( \frac{\log \left( - \log t\right)}{\log d},
  \lambda\right) .
$$
This gives us a conjugacy between $(D,\alpha)$ and $\left(\mathbb R
  \times \mathbb T, \tau \times \beta\right)$ where $\tau(t) = t+1$
and $\beta(\lambda) = \lambda^d$. It follows that $C^*_r\left(\alpha\right) \simeq C^*_r\left(\tau \times \beta\right)$. Let $S^1$ be the
one-point compactification of $\mathbb R$ and let $\tau^+$ be the continuous
extension of $\tau$ to $S^1$. To simplify notation set $\varphi =
\tau^+ \times \beta$. It follows from Proposition 4.6 in
\cite{CT} that there is an extension
\begin{equation}\label{CText}
\begin{xymatrix}{
0 \ar[r]  & C^*_r\left( \tau \times \beta\right) \ar[r] & C^*_r\left(
  \varphi\right) \ar[r] & C^*_r(\beta) \ar[r] & 0.
}\end{xymatrix}
\end{equation}
 In the
notation of Section \ref{dynsys} observe that $R_{\varphi} = S^1
\times R_{\beta}$ and that this decomposition gives rise to an isomorphism
\begin{equation}\label{tensor7}
C^*_r\left(R_{\varphi}\right) \simeq  C\left(S^1\right) \otimes
C^*_r\left(R_{\beta}\right). 
\end{equation}
Under this identification the Deaconu endomorphism $\widehat{\varphi}$ of
$C^*_r\left(R_{\varphi}\right)$ becomes the tensor product
$\hat{\tau} \otimes \widehat{\beta}$ where $\widehat{\beta} :
C^*_r(R_{\beta}) \to C^*_r(R_{\beta})$ is the Deaconu endomorphism of
$C^*_r(\beta)$ and
$\hat{\tau} : C\left(S^1\right)  \to C\left(S^1\right)$ is given by
$$
\hat{\tau}(f)(x) = f(\tau^+(x)) .
$$  
Let $B_{\beta}$ be the inductive limit of the sequence
\begin{equation}\label{parisV}
\begin{xymatrix}{
C^*_r(R_{\beta}) \ar[r]^-{\widehat{\beta}} & C^*_r(R_{\beta})
\ar[r]^-{\widehat{\beta}} & C^*_r(R_{\beta}) \ar[r]^-{\widehat{\beta}} & C^*_r(R_{\beta})
\ar[r]^-{\widehat{\beta}} & \cdots \cdots
}\end{xymatrix}
\end{equation}
and let $\widehat{\beta}_{\infty}$ be the automorphism of $B_{\beta}$
induced by letting $\widehat{\beta}$ act on all copies of
$C^*_r(R_{\beta})$ in the sequence (\ref{parisV}). Similarly, we can consider
the inductive limit $B_{\varphi}$ of the sequence
\begin{equation*}\label{sequne}
\begin{xymatrix}{
C^*_r(R_{\varphi}) \ar[r]^-{\widehat{\varphi}} & C^*_r(R_{\varphi})
\ar[r]^-{\widehat{\varphi}} & C^*_r(R_{\varphi}) \ar[r]^-{\widehat{\varphi}} & C^*_r(R_{\varphi})
\ar[r]^-{\widehat{\varphi}} & \cdots \cdots .
}\end{xymatrix}
\end{equation*}
Using (\ref{tensor7}) and the tensor product decomposition
$\widehat{\varphi} = \hat{\tau} \otimes \widehat{\beta}$ it follows that
$B_{\varphi} \simeq C(S^1) \otimes B_{\beta}$ under an
isomorphism which turns $\widehat{\varphi}_{\infty}$ into $\hat{\tau}
\otimes \widehat{\beta}_{\infty}$. It follows in this way from Theorem 4.8 in \cite{Th1} that there are embeddings of $C^*_r\left(\tau^+ \times
  \beta\right)$ and $C^*_r(\beta)$ into full corners of $\left(C(S^1) \otimes B_{\beta}\right)
\rtimes_{\hat{\tau} \otimes \widehat{\beta}_{\infty}} \mathbb Z$ and
$B_{\beta} \rtimes_{\widehat{\beta}_{\infty}} \mathbb Z$,
respectively. Together with the extension (\ref{CText}) this gives us
a commuting diagram
\begin{equation*}\label{thm48}
\begin{xymatrix}{
0 \ar[r] & C^*_r(\tau \times \beta) \ar[r] \ar[d] & C^*_r\left(\tau^+ \times
  \beta\right) \ar[r] \ar[d] & C^*_r(\beta) \ar[r] \ar[d] & 0 \\ 
0 \ar[r] & \left(C_0(\mathbb R) \otimes B_{\beta}\right)
\rtimes_{\hat{\tau} \otimes \widehat{\beta}_{\infty}} \mathbb Z \ar[r]
& \left(C(S^1) \otimes B_{\beta}\right)
\rtimes_{\hat{\tau} \otimes \widehat{\beta}_{\infty}} \mathbb Z \ar[r]
& B_{\beta} \rtimes_{\widehat{\beta}_{\infty}} \mathbb Z \ar[r] & 0
}\end{xymatrix}
\end{equation*}
with exact rows. Consequently the range of the embedding 
$$C^*_r(\tau \times
\beta) \to  \left(C_0(\mathbb R) \otimes B_{\beta}\right)
\rtimes_{\hat{\tau} \otimes \widehat{\beta}_{\infty}} \mathbb Z
$$ 
is a
full corner, and we conclude therefore from \cite{Br} that $C^*_r\left( \tau \times \beta\right)$ is stably isomorphic to
$\left(C_0(\mathbb R) \otimes B_{\beta}\right) \rtimes_{\hat{\tau} \otimes
  \widehat{\beta}_{\infty}} \mathbb Z$.

Recall that the \emph{mapping torus}
$MT_{\gamma}$ of an endomorphism $\gamma: B \to B$ of a $C^*$-algebra $B$ is the
$C^*$-algebra
$$
MT_{\gamma} = \left\{ f \in C[0,1] \otimes B : \ \gamma(f(0)) = f(1)
\right\} .
$$
We need the following lemma.

\begin{lemma}\label{eee} Let $A$ be a $C^*$-algebra and $\alpha : A
  \to A$ an automorphism. It follows that the crossed product
  $\left(C_0(\mathbb R) \otimes A\right) \rtimes_{\widehat{\tau}
    \otimes \alpha} \mathbb Z$ is isomorphic to the mapping torus
  $MT_{\id_{\mathbb K} \otimes \alpha}$ of
  the automorphism $\id_{\mathbb K} \otimes \alpha : \mathbb K \otimes
  A \to \mathbb K \otimes A$.
\end{lemma}
\begin{proof}  For $f \in C_0(\mathbb R)$ and $n \in \mathbb Z$, let
  $f_n \in C[0,1]$ be the function $f_n(t) = f(n+t)$. Let $e_{ij}, i,j
  \in \mathbb Z$, be
  the standard matrix units in $\mathbb K = \mathbb K(l^2(\mathbb
  Z))$. Define a $*$-homomorphism $\pi : C_0(\mathbb R) \otimes A \to
  C[0,1] \otimes \mathbb K \otimes A$ such that
$$
\pi(f \otimes a) = \sum_{n \in \mathbb Z} f_n \otimes {e_{nn}} \otimes
\alpha^{-n}(a) .
$$
Let $S$ be the two-sided shift on $l^2(\mathbb K)$ such that
$Se_{nn}S^* = e_{n-1,n-1}$. Then $\pi$ maps
into the mapping torus of $(\Ad S) \otimes \alpha$ and $\pi \circ
(\widehat{\tau} \otimes \alpha) = \Ad T \circ \pi$ where $T = 1_{C[0,1]}
\otimes S \otimes 1_A$. It
follows from the universal property of the crossed product that we get
a $*$-homomorphism 
$$
\Phi :\left(C_0(\mathbb R) \otimes A\right) \rtimes_{\widehat{\tau}
    \otimes \alpha} \mathbb Z \to C[0,1] \otimes \mathbb K \otimes
  A
$$ 
which is injective since its restriction to $C_0(\mathbb R) \otimes A$
clearly is. Its range is generated by elements in $C[0,1] \otimes \mathbb K \otimes
  A$ of the form $T\pi(f \otimes a)$ and is therefore contained in the
  mapping torus
  of $(\Ad S) \otimes \alpha$. To see that 
$$
B = \Phi\left(\left(C_0(\mathbb R) \otimes A\right) \rtimes_{\widehat{\tau}
    \otimes \alpha} \mathbb Z\right)
$$ 
actually is equal
  to this mapping torus, let $\ev_t : C[0,1] \otimes \mathbb K \otimes
  A \to \mathbb K \otimes A$ denote evaluation at $t \in [0,1]$. Then
  $\ev_t(B)$ is generated by elements of the form 
$e_{n-1,n} \otimes a$ for some $n \in \mathbb Z$ and some $a
\in A$, and it is easy to see that this is all of $\mathbb K \otimes
A$. Consider then a continuous function $g : [0,1] \to \mathbb K
\otimes A$ which is an element of
$MT_{(\Ad S) \otimes \alpha}$, i.e. has the property that
$$
(\Ad S) \otimes \alpha\left(g(0)\right) = g(1).
$$
Let $\epsilon > 0$. For each $t \in [0,1]$ there is then an element $f_t
\in B$ such that $g(t) = f_t(t)$. We can therefore choose intervals $I_j =
\left[\frac{j}{M},\frac{j+1}{M}\right]$ and elements $f_j \in B$ such
that $\left\|g(t) - f_j(t)\right\| \leq \epsilon, t \in I_j$, for all
$j$, and such that 
$$
\left(\Ad S\right) \circ \alpha \left(f_0(0)\right) = f_{M-1}(1) .
$$
Choose a partition of unity $h_j \in C[0,1], j =
0,1,2,\dots, M-1$, such that $h_0(0) = h_{M-1}(1) = 1$ and $\supp h_j
\subseteq I_j$ for all $j$. Then $f= \sum_{j=0}^{M-1} h_if_i \in B$
because $B$ is a module over $\left\{f \in C[0,1] : \ f(0) =
  f(1)\right\}$. Since $\left\|f-g\right\| \leq \epsilon$ this shows that $B$
is equal to the entire mapping torus of $(\Ad S) \otimes \alpha$. This
mapping torus is isomorphic to that of $\id_{\mathbb K} \otimes \alpha$ because the automorphism group
  of $\mathbb K$ is connected, cf. Proposition 10.5.1 in \cite{Bl}.
\end{proof}

It follows from Lemma \ref{eee} and the preceding considerations that
$C^*_r(\alpha)$ is stably isomorphic to the mapping torus of
$\widehat{\beta}_{\infty} : B_{\beta} \to B_{\beta}$.

It is known that
$C^*_r\left(R_{\beta}\right)$ is isomorphic to the Bunce-Deddens algebra $\BD(d^{\infty})$
of type $d^{\infty}$, cf. Example 3 in \cite{De}. Thus $\BD(d^{\infty})$ is the unique
simple unital AT-algebra with a unique trace state such that
$K_1\left(\BD(d^{\infty})\right) \simeq \mathbb Z$ and $K_0\left(\BD(d^{\infty})\right)$
is isomorphic, as a partial ordered group with order unit, to the group $\mathbb
Z\left[1/d\right]$ of $d$-adic rationals when the latter has the
order inherited from $\mathbb R$ and the order unit $1$. As shown in Example 3 of \cite{De} the map
$\widehat{\beta}_* :  K_1\left(\BD(d^{\infty})\right) \to
K_1\left(\BD(d^{\infty})\right)$ is the identity while $\widehat{\beta}_* :  K_0\left(\BD(d^{\infty})\right) \to
K_0\left(\BD(d^{\infty})\right)$ is multiplication by $\frac{1}{d}$ on $\mathbb
Z\left[1/d\right]$. To emphasise the number $d$, which is the
determining input for the construction, we will denote the mapping torus of
$\widehat{\beta}$ by $MT_d$ in the following.

\begin{lemma}\label{mapping} The mapping torus
  $MT_{\widehat{\beta}_{\infty}}$ of
  $\widehat{\beta}_{\infty}$ is stably isomorphic to
$$
MT_d = \left\{ f \in C[0,1] \otimes C^*_r(\beta) : \ \widehat{\beta}(f(0)) =
  f(1) \right\} ;
$$
the mapping torus of the Deaconu endomorphism of $C^*_r(R_{\beta})$.
\end{lemma}
\begin{proof} Note that the mapping torus of
$\widehat{\beta}_{\infty}$ is isomorphic to the inductive limit
\begin{equation}\label{paris}
\begin{xymatrix}{
MT_d  \ar[rr]^-{\id_{C[0,1]} \otimes
  \widehat{\beta}} & & MT_d \ar[rr]^-{\id_{C[0,1]}
  \otimes \widehat{\beta}}  && MT_d
\ar[rr]^-{\id_{C[0,1]} \otimes \widehat{\beta}}  & & \cdots \cdots
}\end{xymatrix}
\end{equation}
Let $\rho_{\infty,1}
: C_r\left(R_{\beta}\right) \to B_{\beta}$ be the canonical
homomorphism out of the first copy of $C^*_r(R_{\beta})$ in the
sequence (\ref{parisV}). As observed in \cite{An} the isometry $V \in
C^*_r(\beta)$ which implements the Deaconu endomorphism, in the sense
that $\hat{\beta}(a) = VaV^*$, has the property that
$V^*C^*_r(R_{\beta})V \subseteq C^*_r(R_{\beta})$. It follows that
$\hat{\beta}\left(C^*_r(R_{\beta})\right) = VV^*C^*_r(R_{\beta})VV^*$
and that
$$
\rho_{\infty,1}\left(C^*_r\left(R_{\beta}\right)\right) = qB_{\beta}q
$$
where $q = \rho_{\infty,1}(1)$. It follows then from the commuting diagram
\begin{equation*}
\begin{xymatrix}{
MT_d \ar@{^{(}->}[d] \ar[rr] & & MT_{\widehat{\beta}_{\infty}} \ar@{^{(}->}[d] \\
C[0,1]\otimes C^*_r\left(R_{\beta}\right) \ar[rr]^-{\id_{C[0,1]}
  \otimes \rho_{\infty,1}} && C[0,1] \otimes B_{\beta} \\
C^*_r\left(R_{\beta}\right) \ar[u] \ar[rr]_{\rho_{\infty,1}} &&
B_{\beta} \ar[u]  
}
\end{xymatrix}
\end{equation*}
that the image in the mapping torus $MT_{\widehat{\beta}_{\infty}}$ of
the first copy of $MT_d$ from the sequence (\ref{paris}) is equal to
\begin{equation}\label{brown}
\left\{ f \in C[0,1] \otimes B_{\beta} : \
  \widehat{\beta}_{\infty}(f(0)) = f(1), \ qf(t) =
  f(t)q = f(t) \ \forall t \right\},
\end{equation}
which is visibly a hereditary $C^*$-subalgebra of the
mapping cone of $\widehat{\beta}_{\infty}$. Since $B_{\beta}$ is
simple (because $\BD(d^{\infty})$ is) it follows that an ideal in
$MT_{\widehat{\beta}_{\infty}}$ which contains (\ref{brown}) must have
full fiber over every $t \in [0,1]$. Then a standard partition of
unity argument shows, much as in the proof of Lemma \ref{eee}, that such an ideal must be all of
$MT_{\widehat{\beta}_{\infty}}$, i.e. (\ref{brown}) is both hereditary
and full in $MT_{\widehat{\beta}_{\infty}}$. The desired conclusion
follows then from Corollary 2.6 of \cite{Br}.
\end{proof}

We can now summarise with the following.

\begin{prop}\label{super-att} Let $\Omega$ be a super-attractive
  stable region. Then
$C^*_r\left(\Omega \backslash \mathcal I\right)$ is isomorphic to $ \mathbb K
  \otimes MT_d$ where $MT_d$ is the
  mapping torus of the Deaconu endomorphism on
  $\BD\left(d^{\infty}\right)$.
\end{prop}

To describe the quotient $
C^*_r\left({\Omega \cap \mathcal I}\right)$ in 
(\ref{sullivanextension}) we need to determine the restricted
orbits in $\mathcal I$ and find the isotropy groups of their
elements. Note that every periodic point in $\Omega$ is
$\Ro$-equivalent to a critical point in the critical periodic
orbit. Hence every $\Ro$-equivalence class in $\Omega \cap \mathcal I$ is
represented by a critical point $z$ in $\Omega$. When $z$ is
eventually periodic it follows from
Proposition 4.4 of \cite{Th2} that $\Is_z$ is an infinite subgroup of
$\mathbb Q/\mathbb Z$ and hence $C^*_r\left(\Is_z\right) \simeq
C\left( \widehat{\Is_x}\right) \simeq C(K)$
where $K$ is the Cantor set. When $z$ is not pre-periodic it follows
from Proposition 4.4 of \cite{Th2} that $\Is_z = \mathbb Z_v$ where
$v$ is the
  asymptotic valency of $z$. By using Lemma \ref{basdecomp} and Lemma
\ref{discretequot} we get in this way a complete description of
$C^*_r\left({\Omega \cap \mathcal I}\right)$ and we can
then put the information we have obtained into
(\ref{sullivanextension}). To summarise our findings we introduce the
notation $\mathbb K_x$ for the $C^*$-algebra of compact operators on
the Hilbert space $l^2\left(\Ro(x)\right)$. Thus
\begin{equation*}
\mathbb K_x = \begin{cases} \mathbb K \ \text{when $x$ is not
    exposed, and} \\ M_n(\mathbb C),  \ \text{ where $n = \# \Ro(x)
    \leq 4$ when $x$
    is exposed.} \end{cases}
\end{equation*}

\begin{thm}\label{supattdomain} Let $\Omega$ be a super-attractive
  stable region and $c_1,c_2, \dots, c_{n+m}$ critical points in
  $\Omega$ such that $\Omega \cap \mathcal I = \sqcup_{i=1}^{n+m}
  \Ro(c_i)$, and $c_1, c_2, \dots, c_n$ are pre-periodic while $c_{n+1}, c_{n+2},
  \dots, c_{n+m}$ are not. Let $v_i$ be the asymptotic
  valency of $c_i, \ n+1 \leq i \leq n+m$. There is an extension
\begin{equation*}\label{extsup}
\begin{xymatrix}{
0 \ar[r] & \mathbb K \otimes MT_d  \ar[r]
& C^*_r\left({\Omega}\right) \ar[r] & \left( \oplus_{i=1}^n 
C(K) \otimes \mathbb K_{c_i}\right)  \oplus \left(\oplus_{i=n+1}^{n+m} \mathbb
C^{v_i}\otimes \mathbb K_{c_i}\right)   \ar[r] & 0,
}\end{xymatrix}
\end{equation*}
where $K$ is the Cantor set and $MT_d$ is the mapping
torus of the Deaconu endomorphism on the Bunce-Deddens algebra of type $d^{\infty}$.
\end{thm}

\subsection{Attractive stable regions}

Let now $\Omega$ be an attractive stable region. Let $q$ be an element
of the periodic orbit in $\Omega$. The number $\lambda = \left(R^p\right)'(q)$,
where $p$ is the period of $q$, is
the multiplier of $q$. It agrees with the number $\lambda$ from b) of
Definition \ref{fatoudomains}. Let now
$\alpha$ be the local homeomorphism of $D_1$ defined such that
$\alpha(z) = \lambda z$. By the method used in the previous section we find that
$C^*_r\left({\Omega \backslash \mathcal I}\right) \simeq \mathbb K
\otimes C^*_r(\alpha)$. Write $\lambda = |\lambda| e^{2 \pi i \theta}$, where
$\theta \in [0,1[$, so that $\alpha$ can be realised as the map on
$]0,1[ \times \mathbb T$ given by
$$
(t,\mu) \mapsto \left(|\lambda|t, \mu e^{2 \pi i
    \theta}\right) .
$$
Via the map $(t,\mu) \to \left(\frac{\log t}{\log |\lambda|},
  \mu \right)$ we see that $\alpha$ is conjugate to the map $(t,\mu)
  \mapsto \left(t+1,\mu e^{2 \pi i
    \theta}\right)$ on $\mathbb R_+ \times \mathbb T$. The
  transformation groupoid of the last map is a reduction of the
  transformation groupoid of the homeomorphism $(t,\lambda) \mapsto
  \left(t+1,\lambda e^{2 \pi i
    \theta}\right)$ on $\mathbb R \times \mathbb T$. Hence
$C^*_r(\alpha)$ is stably isomorphic to the corresponding crossed product
$C_0(\mathbb R \times \mathbb T) \rtimes \mathbb Z$ by Theorem
\ref{MRW}. It follows from Lemma \ref{eee} that the latter crossed
product is isomorphic to $\mathbb K \otimes C\left(\mathbb
  T^2\right)$. In this way we obtain the following.

\begin{prop}\label{att}
Let $\Omega$ be an attractive stable region. Then 
$$
C^*_r\left({ \Omega \backslash \mathcal I}\right)  \simeq \mathbb K
  \otimes C\left(\mathbb T^2\right).
$$
\end{prop}

It is also straightforward to adopt the methods from the preceding
section to obtain a description of
$C^*_r\left(\Omega \cap {\mathcal I}\right)$. The periodic points lie
in the same restricted orbit and the isotropy group of any of its
members is a copy of $\mathbb Z$ by Proposition 4.4 of \cite{Th2}. The
restricted orbits of the critical points are divided according to
whether or not they are pre-periodic. Since the periodic orbit is not
critical the isotropy group of a critical pre-periodic point is now
$\mathbb Z \oplus  \mathbb Z_d$ where $d$ is the asymptotic valency by
Proposition 4.4 of \cite{Th2}. This leads
to the following description of $C^*_r\left({\Omega}\right)$.

\begin{thm}\label{attdomain} Let $\Omega$ be an attractive stable
  region and $q$ a periodic point in $\Omega$. Let $c_1,c_2, \dots,
  c_{n+m}$ be critical points in
  $\Omega$ such that 
$$\Omega \cap \mathcal I = \Ro(q) \sqcup \sqcup_{i=1}^{n+m} \Ro(c_i),
$$ 
and $c_1, c_2, \dots, c_n$ are
  pre-periodic while $c_i, \ i \geq n+1,$ are not. Let $v_i$ be the asymptotic
  valency of $c_i$. There is an extension

\begin{equation*}\label{extsup}
\begin{xymatrix}{
0 \ar[r] & \mathbb K \otimes C(\mathbb T^2)  \ar[r]
& C^*_r\left({\Omega}\right) \ar[r] &  A  \ar[r]
 & 0
}\end{xymatrix}
\end{equation*}
where
$$
A = \left(C(\mathbb T) \otimes \mathbb
K_q\right)\oplus  \left(\oplus_{i=1}^n \mathbb C^{v_i} \otimes
C(\mathbb T) \otimes \mathbb K_{c_i}\right)  \oplus \left( \oplus_{i=n+1}^{n+m}
  \mathbb C^{v_i}\otimes \mathbb K_{c_i}\right) . 
$$
\end{thm}

\subsection{Parabolic stable regions}

The remaining cases, corresponding to stable regions of parabolic,
Siegel or Herman type can be handled by similar methods. Since the
considerations are simpler than those involved in the attractive
cases, we merely state the results.

\begin{thm}\label{para}
Let $\Omega$ be parabolic stable region. Let $c_i, i = 1,2, \dots,
  N$, be representatives for the restricted orbits of the critical points
  in $\Omega$ and let $v_i$ be the asymptotic valency of $c_i$. There is an extension
\begin{equation*}\label{extpar}
\begin{xymatrix}{
0 \ar[r] & \mathbb K
  \otimes C(\mathbb T) \otimes C_0(\mathbb R)  \ar[r]
& C^*_r\left(\Omega\right) \ar[r] & \oplus_{i=1}^{N}\mathbb
C^{v_i} \otimes \mathbb K_{c_i} \ar[r] & 0 .
}\end{xymatrix}
\end{equation*}
 \end{thm}

\subsection{Stable regions of Siegel type}

Let $\theta \in [0,1] \backslash \mathbb Q$. The corresponding
\emph{irrational rotation algebra} is the universal $C^*$-algebra
generated by two unitaries $U,V$ satisfying the relation $UV =e^{2 \pi
  i \theta}VU$. See \cite{EE} for more on its structure.

\begin{thm}\label{siegel}
Let $\Omega$ be a stable region of Siegel type. Let $q$ be a
periodic point in $\Omega$. Let $c_1,c_2, \dots,
  c_{n+m}$ be critical points in
  $\Omega$ such that 
$$
\Omega \cap \mathcal I = \Ro(q) \sqcup \sqcup_{i=1}^{n+m} \Ro(c_i),
$$ 
and $c_1, c_2, \dots, c_n$ are
  pre-periodic while $c_i, \ i \geq n+1,$ are not. Let $v_i$ be the asymptotic
  valency of $c_i$. There is an extension 
\begin{equation*}\label{extsiegel}
\begin{xymatrix}{
0 \ar[r] & \mathbb K
  \otimes C_0(\mathbb R) \otimes A_{\theta}   \ar[r]
& C^*_r\left(\Omega\right) \ar[r] &  B \ar[r] & 0 
}\end{xymatrix}
\end{equation*}
where $A_{\theta}$ is the irrational rotation algebra, corresponding to the
rotation by the angle $2 \pi \theta$ in the core domain, and
$$
B = \left(C(\mathbb T) \otimes \mathbb
K_q\right)\oplus  \left(\oplus_{i=1}^n \mathbb C^{v_i} \otimes
C(\mathbb T) \otimes \mathbb K_{c_i}\right)  \oplus \left( \oplus_{i=n+1}^{n+m}
 \mathbb C^{v_i}\otimes \mathbb K_{c_i}\right) .
$$
 \end{thm}

\subsection{Stable regions of Herman type}

\begin{thm}\label{herman}
Let $\Omega$ be a stable region of Herman type. Let $c_i, i = 1,2, \dots,
  N$, be representatives for the restricted orbits of the critical points
  in $\Omega$ and let $v_i$ be the asymptotic valency of $c_i$. There is an extension
\begin{equation*}\label{extherman}
\begin{xymatrix}{
0 \ar[r] & \mathbb K
  \otimes C_0(\mathbb R) \otimes A_{\theta}   \ar[r]
& C^*_r\left(\Omega\right) \ar[r] & \oplus_{i=1}^{N} \mathbb
C^{v_i} \otimes \mathbb K_{c_i} \ar[r] & 0 
}\end{xymatrix}
\end{equation*}
where $A_{\theta}$ is the irrational rotation algebra corresponding to the
rotation by the angle $2 \pi \theta$ in the core domain. 
 \end{thm}

\subsection{A square of six extensions}\label{secdiag}

It is possible to combine the extensions from the last sections into
an exact square of 6 extensions in the following way. Let $\mathcal
I_p$ be the union of the $\Ro$-orbits containing a non-critical
periodic orbit in $F_R$ and $\mathcal I_c = \mathcal I
\backslash \mathcal I_p$ its complement in $\mathcal I$. Several applications of Lemma \ref{quot}
gives us the following commuting diagram with exact rows and columns.

\begin{equation}\label{introdiag2}
\begin{xymatrix}{ 
  & 0 \ar[d]  & 0 \ar[d] & 0 \ar[d]  & \\
0 \ar[r] &  C^*_r\left(F_R \backslash \mathcal I\right)  \ar[r] \ar[d] &
  C^*_r\left(F_R \backslash \mathcal I_c\right) \ar[r] \ar[d]  &
 C^*_r\left(\mathcal I_p\right) \ar[r] \ar[d]
& 0 \\
0 \ar[r] &   C^*_r\left(F_R \backslash \mathcal I_p\right)  \ar[r]
\ar[d] &
C^*_r\left(R\right)  \ar[r] \ar[d] &  C^*_r\left(J_R \cup \mathcal I_p\right) \ar[d] \ar[r]
& 0 \\
0 \ar[r] &   C^*_r\left(\mathcal I_c\right)  \ar[r]
\ar[d] &  C^*_r\left(J_R \cup\mathcal I_c\right)  \ar[r] \ar[d] & C^*_r(J_R) \ar[d] \ar[r]
& 0 \\
  & 0   & 0  & 0   & \\
}\end{xymatrix}
\end{equation} 
The algebras in the corners, $C^*_r\left(F_R \backslash \mathcal
  I\right)$, $C^*_r\left(\mathcal I_p\right)$, $C^*_r\left(\mathcal
  I_c\right)$ and $C^*_r(J_R)$, can all be identified from the
preceding sections. Specifically, $C^*_r(J_R)$ is either nuclear,
simple and purely infinite, or an extension of such an algebra by a
finite direct sum of circle and matrix algebras, cf. Theorem \ref{Juliaext}. $C^*_r\left(F_R \backslash \mathcal
  I\right)$ is a finite direct sum of algebras each of which is the
stabilization of $MT_d$, $C\left(\mathbb T^2\right)$, $C(\mathbb T)
\otimes C_0(\mathbb R)$ or $C_0(\mathbb R) \otimes A_{\theta}$. Which
of the four types are present depends on the nature of the
stable regions in $F_R$. The
algebra $C^*_r\left(\mathcal I_p\right)$ is a finite direct sum of
algebras stably isomorphic to $C(\mathbb T)$ while
$C^*_r\left(\mathcal I_c\right)$ is a finite direct sum of algebras
stably isomorphic to $\mathbb C$, $C(\mathbb T)$ or $C(K)$. Which
summands occur depends on the behaviour under iteration of the
critical points in $F_R$.

It should be noted that the decomposition of $C^*_r(R)$ depicted in
(\ref{introdiag2}) is not the only possible. In fact there is a commuting square of the form
(\ref{introdiag2}) for any $\Ro$-invariant partitioning of $\mathcal
I$; not just for the partition $\mathcal I = \mathcal I_p \sqcup
\mathcal I_p$ chosen above.

%\begin{remark} It follows from the structure of $\F$ we have unravelled in this section that this algebra is nuclear and
%  satisfies the UCT. In combination with Theorem \ref{Juliaext} and
%  Corollary \ref{jf-cor} it follows that the same is true for $\J$ and
%  $C^*_r(R)$.
%\end{remark}

\section{Primitive ideals and primitive quotients}

In the following an ideal in a $C^*$-algebra is a closed two-sided and proper ideal. Recall that an ideal $I$ is
\emph{primitive} when it is the kernel of an irreducible non-zero
representation, and \emph{prime} when it has the property that $I_1I_2
\subseteq I \Rightarrow I_1 \subseteq I$ or $I_2 \subseteq I$ when
$I_1$ and $I_2$ are also ideals. Since we shall only deal with
separable $C^*$-algebras the primitive ideals will be the same as the
prime ideals, cf. e.g. \cite{RW}.

\subsection{The primitive ideals}

When $I$ is an ideal in $C^*_r\left(R\right)$ we set
$$
\rho(I) = \left\{ x \in \C : \ f(x) = 0 \ \forall f \in C(\C) \cap I
\right\} .
$$
We call $\rho(I)$ the \emph{co-support} of $I$.

\begin{lemma}\label{cloinv} $\rho(I)$ is a closed non-empty $\Ro$-invariant
  subset of $\C$.
\end{lemma}
\begin{proof} See Lemma 4.5 in \cite{CT}.
%
%$\rho(I)$ is closed by definition. If empty it would
%  imply that $C(\C) \subseteq I$ which is impossible because $C(\C)$
%  contains an approximate unit for $C^*_r(R)$ and $I$ is a proper
%  ideal. Let $[x,k,\eta,y]$ be
%  an element of $G_R$ such that $y \in \rho(I)$. There is then an open subset $V$ of
%  $x$ such that 
%$$
%W = \left\{ [z,k,\eta,\eta(z)] : \ z \in V \right\}
%$$
%is an open neighbourhood of $[x,k,\eta,y]$ in $G_R$. Let $h \in
%C_c\left(G_R\right)$ be supported in $W$ such that $h[x,k,\eta,y] =
%1$. If $f \in C(\C)\cap I$ 
%and $f(x) \neq 0$, we find that $h^*fh \in C(\C) \cap I$ and $h^*fh(y) \neq
%0$, a contradiction. It follows that $f(x) = 0$ for all $f \in C(\C)
%\cap I$, i.e. $x \in \rho(I)$.  
\end{proof} 

\begin{lemma} \label{ideal-gen}
  Let $I$ be an ideal in $C_r^*(R)$ and let
  $A$ be a closed $\Ro$-invariant subset of $\C$.
  If $\rho(I)\subseteq A$, then $\ker\pi_A\subseteq I$.
\end{lemma}

\begin{proof} See Lemma 4.8 in \cite{CT}.
%Since $\rho(I)\subseteq A$ it follows from
%  the Stone-Weierstrass theorem that $C_0(\C \setminus A)\subseteq I
%  \cap C(\C)$. Let
%  $\left\{i_n\right\}$ be an approximate unit in $C_0(\C \backslash
%  A)$. It is easy to see that $\{i_n\}$ is also an approximate unit in
%  $C^*_r\left({\C \backslash A}\right)$ and then by Lemma \ref{quot} also in $\ker \pi_A$. Since $\{i_n\} \subseteq
%  I$ it follows that $\ker\pi_A\subseteq I$.  
\end{proof}

When $I \subseteq C^*_r(R)$ is an ideal we let $q_I : C^*_r(R) \to
C^*_r(R)/I$ denote the corresponding quotient map.
Note that it follows from Lemma \ref{ideal-gen} that $q_I$ factorises
through $C^*_r\left(\rho(I)\right)$, i.e. there is a $*$-homomorphism
$C^*_r\left(\rho(I)\right) \to C^*_r(R)/I$ such that
\begin{equation}\label{triangle}
\begin{xymatrix}{
C^*_r(R) \ar[dr]_{\pi_{\rho(I)}} \ar[rr]^{q_I} & &C^*_r(R)/I \\
& C^*_r\left(\rho(I)\right) \ar[ur] & 
}
\end{xymatrix}
\end{equation}
commutes.

A non-empty closed $\Ro$-invariant subset $A \subseteq \C$ is
\emph{prime} when the implication
$$
A \subseteq B \cup C \ \Rightarrow \ A \subseteq B \ \text{or} \ A
\subseteq C 
$$
holds for all closed $\Ro$-invariant subsets $B$ and $C$ of $\C$.

\begin{lemma} \label{prime}
Assume that $I$ is a primitive ideal in $C_r^*(R)$. It follows that
$\rho(I)$ is prime.
\end{lemma}
\begin{proof} 
See Proposition 4.10 in \cite{CT}.
%Assume that $B$ and $C$ are
%  closed $\Ro$-invariant subsets such that $\rho(I)\subseteq
%  B\cup C$. Then
%\begin{equation}\label{ideleq}
%\begin{split}
%&\ker \pi_B \cap \ker \pi_C \cap C(\C) =
%  C(\C \backslash B) \cap C(\C \backslash C)\\
%& = C_0\left(\C \backslash
%  (B\cup C)\right) =  \ker \pi_{B \cup C} \cap C(\C).
%\end{split}
%\end{equation} 
%Since $\ker \pi_{B \cup C} = C^*_r\left({\C \backslash (B \cup C)}\right)$ by Lemma \ref{quot} it follows that $\ker \pi_{B \cup C} \cap C(\C)$ contains
%an approximate unit for $\ker \pi_{B \cup C}$. The same lemma implies also, in two
%steps, that $\ker \pi_B \cap \ker \pi_C \cap C(\C)$ contains an
%approximate for $\ker \pi_B \cap \ker \pi_C$. Therefore (\ref{ideleq})
%implies that $\ker \pi_B \cap \ker \pi_C = \ker \pi_{B \cup C}$. Since
%$\ker(\pi_{B\cup C})\subseteq I$ by Lemma \ref{ideal-gen}
%the primeness of $I$ implies that $\ker \pi_B \subseteq I$ or $\ker \pi_C
%\subseteq I$. It follows from Lemma \ref{quot} that $B = \rho(\ker
%\pi_B)$ and $C = \rho(\ker \pi_C)$. Thus $\rho(I) \subseteq B$ or
%$\rho(I) \subseteq C$.
\end{proof}

\begin{lemma}\label{prime2} Let $Y$ be a prime subset of $\C$. Assume that $x \in Y$ is isolated in $Y$. Then all elements
  of $\Ro(x)$ are isolated in $Y$ and
$Y = \overline{\Ro(x)}$.
\end{lemma}
\begin{proof} It is clear that all elements of $\Ro(x)$ are isolated in
  $Y$ since $x$ is. Set 
$$
B = \overline{ \left\{z \in Y :  \ z \notin {\Ro(x)} \right\}} .
$$ 
Since $Y \subseteq \overline{\Ro(x)} \cup B$ the primeness of $Y$
implies $Y \subseteq \overline{\Ro(x)}$ or $Y \subseteq B$. Note that $x
\notin B$ since $x$ is isolated in $Y$. It follows that $Y  \subseteq \overline{\Ro(x)}$. 
\end{proof}

In the following we denote by $\orb(x)$ the (full) orbit of $x$, i.e. 
$$
\orb(x) = \left\{y \in \C : \ R^n(x) = R^m(y) \ \text{for some} \ n,m
  \in \mathbb N \right\}.
$$

\begin{lemma}\label{prime3}  Let $Y$ be a prime subset of $\C$. Assume $Y$ has no isolated points. It
  follows that there is a point $x \in Y \backslash \bigcup_{j
    =0}^{\infty} R^{-j}(\crt)$ such that
$Y =  \overline{\Ro(x)} = \overline{\orb(x)}$.
\end{lemma}
\begin{proof} The proof is largely the same as the proof of Proposition 4.9
  in \cite{CT}, but with a few crucial modifications. It follows from
  Lemma \ref{escape} that $Y$ is totally $R$-invariant and
  hence in particular that $\overline{\Ro(x)} \subseteq \overline{\orb(x)} \subseteq Y$ for all
  $x \in Y$. It suffices therefore to find an $x \in Y \backslash \bigcup_{j
    =0}^{\infty} R^{-j}(\crt )  $ such that $Y \subseteq
  \overline{\Ro(x)}$. Let $\{U_k\}_{k=1}^\infty$ be a basis for the topology of
  $Y$. We will by induction construct
  compact sets $\{C_k\}_{k=0}^\infty$ and
  $\{C_k'\}_{k=0}^\infty$ with non-empty interiors in $Y$ and positive integers $(n_k)_{k=0}^\infty$
  and $(n_k')_{k=0}^\infty$ such that 
\begin{itemize}
\item[i)] $C_k\subseteq U_k$,
\item[ii)] $C_k'\subseteq R^{n_{k-1}}(C_{k-1})\cap
  R^{n'_{k-1}}\left(C'_{k-1} \right)$ when $k \geq 1$ and
\item[iii)] $C'_k \cap \left( \bigcup_{j=0}^{n'_0 + n'_1 + \cdots +
      n'_{k-1}} R^j(\crt) \cup  \bigcup_{j=0}^{n_{k-1}} R^j(\crt)\right)   = \emptyset$ when $k \geq 1$.
\end{itemize}
Let $C_0=C'_0$ be any compact subset of $Y$ with non-empty
interior in $Y$. Assume that $k\ge 1$ and that
  $C_1,\dots,C_k$, $C'_1,\dots,C'_k$, $n_0,\dots,n_{k-1}$ and $n'_0,\dots,n'_{k-1}$
  satisfying the conditions above have been chosen. Choose non-empty
  open subsets $V_k\subseteq C_k$ and $V'_k\subseteq C'_k$.
  Then 
  \begin{equation*}
    \bigcup_{l,m=0}^\infty R^{-l}(R^m(V_k))\text{ and
    }\bigcup_{l,m=0}^\infty R^{-l}(R^m(V'_k))
  \end{equation*}
  are non-empty open and totally $R$-invariant
  subsets of $Y$, and hence 
  \begin{equation} \label{eq:1}
    Y\setminus\bigcup_{l,m=0}^\infty R^{-l}(R^m(V_k))\text{ and
    }Y\setminus\bigcup_{l,m=0}^\infty R^{-l}(R^m(V'_k)) 
  \end{equation}
  are closed and totally $R$-invariant subsets of $Y$. Since $Y$ is prime and not
  contained in either of the sets from \eqref{eq:1}, it is also not
  contained in their union. That is,
%  $Y$ is not contained in 
%  \begin{equation*} 
%    \left(H\setminus\bigcup_{l,m=0}^\infty H^{-l}(H^m(V_k))\right)\bigcup
%    \left(H\setminus\bigcup_{l,m=0}^\infty H^{-l}(H^m(V'_k)) \right)
%  \end{equation*}
%  and thus that 
  \begin{equation*} 
    \left(\bigcup_{l,m=0}^\infty R^{-l}(R^m(V_k))\right)\bigcap
    \left(\bigcup_{l,m=0}^\infty R^{-l}(R^m(V'_k))\right) \ne\emptyset.
  \end{equation*}
  It follows that there are positive integers $n_k$ and $n'_k$
  such that $R^{n_k}(V_k)\cap R^{n'_k}(V'_k)$ is
  non-empty. We can therefore choose a non-empty compact set $C_{k+1}$ with
  non-empty interior such that $C_{k+1}\subseteq
  U_{k+1}$ and a non-empty compact set $C'_{k+1}$ with non-empty
  interior such that $C'_{k+1}\subseteq
  R^{n_k}(V_k)\cap R^{n'_k}(V'_k)$. Since 
$$
\bigcup_{j=0}^{n'_0 + n'_1 + \cdots +
      n'_{k}} R^j(\crt) \cup  \bigcup_{j=0}^{n_{k}} R^j(\crt)
$$ 
is a finite set and $Y$ contains no isolated points,
  we can arrange that
\begin{equation*}\label{delete}
C'_{k+1} \cap \left(\bigcup_{j=0}^{n'_0 + n'_1 + \cdots +
      n'_{k}} R^j(\crt) \cup  \bigcup_{j=0}^{n_{k}} R^j(\crt) \right)  = \emptyset.
\end{equation*}
This completes the induction step.

It follows from ii) that 
$$
C_k' = R^{n'_0 + n'_1 + \dots + n'_{k-1}} \left( C'_0 \cap R^{-n'_0}
  (C'_1) \cap \dots \cap R^{-n'_0 - n'_1 - \dots -
    n'_{k-1}}\left(C'_k\right)\right) 
$$
for all $k$ and hence that
  \begin{equation*}
    C'_0\cap R^{-n'_0}(C'_1)\cap\dots \dots \cap R^{-n'_0- \dots -n'_k}(C'_{k+1}),\ k=0,1,\dots
  \end{equation*}
  is a decreasing sequence of non-empty compact sets. Let 
  \begin{equation*}
    x\in \bigcap_{k=0}^\infty R^{-n'_0-\dots \dots -n'_k}(C'_{k+1})\cap C'_0.
  \end{equation*}
By construction there is for each $k$ an element $u_k \in U_k$ such
that $R^{n'_0 + \dots +n'_k}(x) = R^{n_k}(u_k)$ and 
$$
\val \left(R^{n'_0 + \dots +n'_k},x\right) = \val\left(R^{n_k},
  u_k\right) = 1 .
$$
Since this implies that $u_k \in \Ro(x)$ we conclude that $\Ro(x)$ is
dense in $Y$. Furthermore, it implies also that $\val
\left(R,R^j(x)\right) = 1$ for all $j$, i.e. $x \notin \bigcup_{j
    =0}^{\infty} R^{-j}(\crt ) $. 

\end{proof}

\begin{cor}\label{onegenerator} Let $Y \subseteq \C$ be a closed
  $\Ro$-invariant subset. Then $Y$ is prime if and only if there is a
  point $x \in \C$ such that $Y = \overline{\Ro(x)}$.
\end{cor}
\begin{proof} If $Y$ is prime it follows from Lemma \ref{prime2} and
  Lemma \ref{prime3} that there is an element $x \in \C$ such that $Y =
  \overline{\Ro(x)}$. This proves the necessity of the
  condition. Sufficiency follows immediately from the definitions.
\end{proof}

Let $\mathcal M$ be the set of prime subsets of
$\C$. Let $\mathcal M_{ex}$ denote the collection of elements $Y \in \mathcal M$ with the property that $Y$ contains an
isolated point which is either periodic or critical.

\begin{lemma}\label{prime4} Let $Y \in \mathcal M \backslash \mathcal
  M_{ex}$. It follows that $\ker \pi_Y$ is the only ideal $I$ in
  $C^*_r\left(R\right)$ with $\rho(I) = Y$, and that $\ker \pi_Y$ is
  a primitive ideal in $C^*_r\left(R\right)$.
\end{lemma}
\begin{proof}
 Let $I$ be an ideal in $C^*_r\left(R\right)$ with
$\rho(I) = Y$. Then $\ker \pi_Y \subseteq I$ by Lemma
\ref{ideal-gen}. To conclude that $I = \ker \pi_Y$ it suffices
therefore to show that $\pi_Y(I) = \{0\}$ in
$C^*_r\left(Y\right)$. To this end note first that $\pi_Y(I)
\cap C(Y) = \{0\}$. Indeed, if $h \in \pi_Y(I) \cap C(Y)$, let $g
\in C(\C)$ be a function such that $g|_Y = h$ and let $a \in I$ be an
element such that $\pi_Y(a) = h$. Then $\pi_Y(a-g) = 0$ and hence $a
-g \in \ker \pi_Y \subseteq I$. It follows that $ g= a - (a-g) \in I
\cap C(\C)$ and hence that $g(y) =0$ for all $y \in \rho(I) =
Y$. Thus $h = 0$, proving that $\pi_Y(I)
\cap C(Y) = \{0\}$. To conclude from this that $\pi_Y(I) = 0$ note
first that the elements of $Y$ with non-trivial isotropy in
$G_Y$ are dense in $Y$. This follows from Lemma \ref{7!!} because a point $y \in Y$
with non-trivial isotropy in $G_Y$ must be pre-periodic or
pre-critical for $R$ by Proposition 4.4 a) in \cite{Th2}. It follows then from Lemma 2.15 of
\cite{Th1} that $P\left(\pi_Y(I)\right) = \{0\}$ when $P :
C^*_r\left(Y\right) \to C(Y)$ denotes the
conditional expectation. Since $P$ is faithful this shows that $\pi_Y(I) =
0$ and hence that $I = \ker \pi_Y$.

To show that $\ker \pi_Y$ is primitive we may as well show that
$C^*_r\left(Y\right)$ is a prime $C^*$-algebra. Consider
therefore two ideals $I_j \subseteq
C^*_r\left(Y\right), j = 1,2$, such that $I_1I_2 
= \{0\}$. Then
$$
\left\{ y \in Y : f(y) = 0 \ \forall f \in I_1 \cap C(Y) \right\}
\cup \left\{ y \in Y : f(y) = 0 \ \forall f \in I_2 \cap C(Y) \right\}  = Y .
$$ 
By Corollary \ref{onegenerator} there is an element $x \in Y$ such
that $Y =
\overline{\Ro(x)}$. Then $x$ must be in $\left\{ y
  \in Y : f(y) = 0 \ \forall f \in I_j \cap C_0(Y) \right\}$ for
either $j = 1$ or $j=2$. Assume without loss of generality that $x \in
\left\{ y \in Y : f(y) = 0 \ \forall f \in I_1 \cap C(Y)
\right\}$. The latter set is both closed and $\Ro$-invariant so we
conclude that 
$$
Y = \overline{\Ro(x)} = \left\{ y \in Y : \ f(y) = 0 \ \forall
  f \in I_1 \cap C(Y) \right\},
$$ 
i.e $I_1 \cap C(Y) = \{0\}$. As above we conclude from this that
$I_1 = \{0\}$, thanks to Lemma 2.15 of \cite{Th1}.
\end{proof}

Let $Y \in \mathcal M_{ex}$ and let $y \in Y$ be an isolated point
which is either periodic or critical. By
Proposition 4.4 of \cite{Th2} the isotropy group $\Is_y$ is abelian and in fact either
$\mathbb Z$, a non-zero subgroup of $\mathbb Q/\mathbb Z$ or isomorphic to
$\mathbb Z \oplus \mathbb Z_d$ for some $d \in \mathbb N$. Let
$\widehat{\Is_y}$ be its Pontryagin dual group. Since $y$ is isolated in $Y$ every element $\xi \in \Is_y$ is isolated
in $G_Y$ and hence the characteristic function $1_{\xi}$ of the
set $\{\xi\}$ is an element of $C_c\left(G_Y\right)
\subseteq C^*_r\left(Y\right)$. For each
$\omega \in \widehat{\Is_y}$ set 
$$
I(y,\omega) =
\pi_Y^{-1}(I_0(y,\omega)),
$$ 
where $I_0(y,\omega)$ is the
ideal in $C^*_r\left(Y\right)$ generated by the elements
$$
1_{[y,0,\id,y]} - \overline{\omega(\xi)} 1_{\xi} , \ \xi 
\in {\Is_y} .
$$ 
By adopting the proof of Proposition 4.15 from \cite{CT} in a
straightforward way we obtain the following.

\begin{lemma}\label{excep} Let $Y \in \mathcal M_{ex}$ and let $y \in
  Y$ be an isolated point. Then the map
$\widehat{\Is_y} \ni \omega \mapsto I(y,\omega)$
is a bijection from $\widehat{\Is_y}$ onto the collection of primitive ideals $I$
in $C^*_r\left(R\right)$ with $\rho(I) = Y$.
\end{lemma}

In particular, it follows from Lemma \ref{excep} and Lemma
\ref{prime4} that every prime subset of $\overline{\mathbb C}$ is the
co-support of a primitive ideal in $C^*_r(R)$. By combining Lemma \ref{excep} with Lemma \ref{prime4} we get the following.

\begin{lemma}\label{primitive} For each $A \in \mathcal M_{ex}$ choose
  an isolated point $y_A$ in $A$ which is either periodic or critical. Then the set of primitive ideals in
  $C^*_r(R)$ is the disjoint union
$$
\left\{\ker \pi_B : \ B \in \mathcal M \backslash \mathcal M_{ex}
\right\} \cup \bigcup_{A \in \mathcal M_{ex}} \left\{ I(y_A,\omega) : \
  \omega \in \widehat{\Is_{y_A}}\right\} .
$$
\end{lemma}

\begin{lemma}\label{mex} Let $A \in \mathcal M_{ex}$. There is either an exposed point $x \in E_R$ such that $A = \Ro(x)$, or a
  critical or periodic point $x \in F_R \backslash E_R$ such that $A = \Ro(x) \sqcup
  J_F$.
\end{lemma}
\begin{proof} Let $x$ be periodic or critical point such that $x$ is
  isolated in $A$ and $A =
  \overline{\Ro(x)}$. If $x \in J_R$ it follows from Lemma \ref{snit}
  that $\Ro(x)$ is finite since $J_R$ has no isolated points, i.e. $x$
  is exposed. Assume $x \in F_R$. It follows from Lemma
  \ref{perclosed} and Lemma \ref{gcclosed} that $\overline{\Ro(x)}
  \backslash \Ro(x) \subseteq J_R$. Since $\overline{\Ro(x)}
  \backslash \Ro(x)$ is closed, $\Ro$-invariant and has no isolated
  points, it follows then from Lemma \ref{snit} that $\overline{\Ro(x)}
  \backslash \Ro(x) = \emptyset$ or $\overline{\Ro(x)}
  \backslash \Ro(x) = J_R$. In the first case $x$ is exposed and in
  the second we have that $A = \Ro(x) \sqcup
  J_F$.
\end{proof}

\begin{lemma}\label{nonmex} Let $A \in \mathcal M \backslash \mathcal
  M_{ex}$. Then either $A = J_R$ or $A = \overline{\Ro(x)}$ for some
  $x \in F_R \backslash \mathcal I$. In the last case, $J_R \subseteq A$.
\end{lemma}
\begin{proof} Let $x \in A$ such that $\overline{\Ro(x)} = A$. If $x
  \in J_R$ it follows from Lemma \ref{snit} that $A = J_R$ or $A$ is
  finite. In the last case $A = \Ro(x)$ is an exposed $\Ro$-orbit
  which must contain either a periodic point or a critical point,
  cf. Section \ref{stabjr}. Since this is impossible when $A \notin
  \mathcal M_{ex}$ we must have that $A = J_R$.

Assume that $x \in F_R$. If $A$ contains an isolated point $y$ it
follows form Lemma \ref{prime2} that $ A = \overline{\Ro(y)}$. Note
that $y \in F_R$ because $y \in \Ro(x)$ and $F_R$ is totally
$R$-invariant. It follows that $y \notin \mathcal I$ since $A \notin
\mathcal M_{ex}$. To prove that $J_R
\subseteq A$ note that $R^n(y) \in \Ro(y)$ for all $n \in
\mathbb N$ since $y \notin
\bigcup_{i=0}^{\infty} R^{-i}\left(\crt \right)$. Thus $y$ can not be
pre-periodic since this would contradict that $A \notin \mathcal
M_{ex}$. Now an argument from the proof of Lemma 7.3 in \cite{Th2}
shows that there is an $n \in \mathbb N$ such that the backward orbit
of $R^n(y)$ contains no critical points. Then the backward orbit of
$R^n(y)$ is contained in $\Ro(y)$, and since $R^n(y)$ is not
exceptional, it follows therefore from Theorem 4.2.5 in \cite{B} that
$J_R \subseteq \overline{\Ro(x)}$.

If $A$ has no isolated points it follows from Lemma \ref{prime3} that
there is a point $y \in A\backslash \bigcup_{j=0}^{\infty}
R^{-j}(\crt)$ such that $A = \overline{\Ro(y)}$. Note that $y$ can not
be pre-periodic because $A \notin
\mathcal M_{ex}$. Since the asymptotic valency of $y$ is $1$ it
follows that $y \notin \mathcal I$, i.e. $y \in F_R \backslash \mathcal
I$. As above it follows that there is an $n \in \mathbb N$ such that
the backward orbit of $R^n(y)$ is in $\Ro(y)$ and it follows again
that $J_R \subseteq A$.

\end{proof}

We can now show that the primitive ideal space of $C^*_r(R)$ is not
Hausdorff, or even $T_0$, in the hull-kernel topology unless $J_R = \C$ and there are no
exposed points. Indeed, if $J_R \neq \C$ it follows from Lemma
\ref{7!!} that $F_R$
contains a point $y$ which is neither pre-critical nor
pre-periodic or exposed. Then Lemma \ref{mex} combined with Lemma \ref{gcclosed} and
Lemma \ref{perclosed} shows that $A =\overline{\Ro(y)} \in \mathcal M
\backslash \mathcal M_{ex}$. Furthermore, $J_R \subseteq A$ by Lemma \ref{nonmex}. When $J_R = \C$ and there is an exposed point, its
restricted orbit will be an element $B \in
\mathcal M_{ex}$ such that $B \subsetneq J_R$. In the first case it
follows from Lemma \ref{ideal-gen} that $\ker \pi_{A}
\subsetneq \ker \pi_{J_R} $ so that $\ker \pi_{J_R}$ is in the closure of
$\left\{\ker \pi_{A}\right\}$ with respect to the hull-kernel
topology. In the second case $\{0\} = \ker \pi_{J_R}  \subsetneq \ker \pi_B
\subseteq I(y_B,\omega)$  for any $y_B \in B$ and any $\omega \in
\widehat{\Is_{y_B}}$, and then $I(y_B,\omega)$ is a primitive ideal in the closure of
$\left\{\ker \pi_{J_R}\right\}$. In both cases we conclude that the
primitive ideal spectrum is not $T_0$. Note that when $J_R = \C$
and there are no exposed points, $C^*_r(R)$ is simple by Proposition
\ref{JRsimple} and the primitive ideal spectrum reduces to one point.

\subsection{The primitive quotients}

It follows from Lemma \ref{mex} and Lemma \ref{nonmex} that we can
divide the primitive ideals $I$ of $C^*_r(R)$ into four types,
according to the nature of their co-supports: 
\begin{enumerate}
\item[i)] $\rho(I) = J_R$,
\item[ii)] $\rho(I) = \Ro(x)$ for some exposed point $x$,
\item[iii)] $\rho(I) = \Ro(x) \sqcup J_R$ for some $x \in \mathcal I
  \cap F_R
  \backslash E_R$, and
\item[iv)] $\rho(I) = \overline{\Ro(x)}$ for some $x \in F_R \backslash
  \mathcal I$.
\end{enumerate}

When $\rho(I) = J_R$ the quotient $C^*_r(R)/I$ is $C^*_r(J_R)$ whose
structure was elucidated in Section \ref{JR}. When $\rho(I) = \Ro(x)$ for some
exposed point it follows from (\ref{triangle}), Lemma
\ref{discretequot} and Corollary \ref{four} that $C^*_r(R)/I \simeq
M_n(\mathbb C)$ for some $n \leq 4$. In case iii) it follows first
from Lemma \ref{quot} and Lemma \ref{discretequot} that there is an extension 
\begin{equation*}\label{3x0}
\begin{xymatrix}{ 
0 \ar[r] &  C^*(\Is_x) \otimes \mathbb K  \ar[r]
 &
C^*_r(\rho(I))   \ar[r]   & C^*_r\left(J_R \right)  \ar[r] & 0
}
\end{xymatrix}
\end{equation*} 
and then from (\ref{triangle}) and Lemma \ref{excep} that there is an extension 
\begin{equation*}\label{3y0}
\begin{xymatrix}{ 
0 \ar[r] &  \mathbb K  \ar[r]
 &
C^*_r(R)/I   \ar[r]   & C^*_r\left(J_R \right)  \ar[r] & 0.
}
\end{xymatrix}
\end{equation*} 

It remains to describe the primitive quotient $C^*_r(R)/I$ in case
iv). The result depends very much on which stable region the point $x
\in F_R \backslash \mathcal I$ which generates $\rho(I)$
comes from. We consider the different possibilities in the following subsections.

\subsubsection{The super-attractive and attractive stable regions}

Assume $x $ is contained in a super-attracting stable region $\Omega$. It follows from Lemma
\ref{prime4} that $C^*_r(R)/I \simeq C^*_r(\rho(I))$. Since $I \cap
C^*_r\left(\Omega \backslash \mathcal I\right)$ is a primitive ideal
in $C^*_r\left(\Omega \backslash \mathcal I\right)$ it follows that
$$
C^*_r\left(\rho(I) \cap \Omega \backslash \mathcal I\right) =
C^*_r\left(\Omega \backslash \mathcal I\right)/I
$$ 
is a primitive
quotient of $C^*_r(\Omega \backslash \mathcal I)$ and hence isomorphic to
the stabilised Bunce-Deddens algebra $\mathbb K \otimes \BD\left(d^{\infty}\right)$ by Proposition
\ref{super-att}. When we apply the method from Section
\ref{secdiag} to $C^*_r(\rho(I))$ rather than $C^*_r(R)$, we obtain
therefore the
following commuting diagram with exact rows and columns because there
are no periodic non-critical orbits.    
\begin{equation*}\label{3x4}
\begin{xymatrix}{ 
  & 0 \ar[d]  & 0 \ar[d] &   & \\
  &  \BD(d^{\infty}) \otimes \mathbb K \ar@{=}[r] \ar[d] &
  \BD(d^{\infty}) \otimes \mathbb K   \ar[d]  & &  \\
0 \ar[r] &  C^*_r\left(\rho(I) \cap F_R\right) \ar[r]
\ar[d] &
C^*_r\left(R\right)/I  \ar[r] \ar[d] & C^*_r(J_R)  \ar@{=}[d] \ar[r]
&  0 \\
0 \ar[r] &  A  \ar[r]
\ar[d] &
 C^*_r\left(J_R \cup  \left( \mathcal I_c \cap \rho(I)\right)\right)   \ar[r] \ar[d] & C^*_r\left(J_R \right)  \ar[r]
& 0 \\
  & 0   & 0  &    & \\
}
\end{xymatrix}
\end{equation*}
Here
$$
A = C^*_r\left(\mathcal I_c \cap \rho(I)\right) = \left( \oplus_{i=1}^n 
C(K) \otimes \mathbb K_{c_i}\right)  \oplus \left(\oplus_{i=n+1}^{n+m} \mathbb
C^{v_i}\otimes \mathbb K_{c_i}\right)
$$
where $c_1,c_2, \dots, c_{n+m}$ are critical points in
  $\Omega$ such that $\rho(I) \cap \mathcal I_c = \sqcup_{i=1}^{n+m}
  \Ro(c_i)$, and $c_1, \dots, c_n$ are pre-periodic while $c_{n+1},
  \dots, c_{n+m}$ are not. As usual $v_i$ is the asymptotic
  valency of $c_i$ and $K$ is the Cantor set.

When $\Omega$ is attractive with periodic point $p$ we get by the same
reasoning the diagram
\begin{equation*}\label{3x5}
\begin{xymatrix}{ 
  & 0 \ar[d]  & 0 \ar[d] & 0 \ar[d]  & \\
0 \ar[r] & \mathbb K \ar[r] \ar[d] & C^*_r\left(\Ro(p)\right)    \ar[r] \ar[d]  &
C(\mathbb T)\otimes \mathbb K_p \ar[r] \ar[d]
& 0 \\
0 \ar[r] &  C^*_r\left(\rho(I) \cap F_R \backslash \Ro(p)\right) \ar[r]
\ar[d] &
C^*_r\left(R\right)/I  \ar[r] \ar[d] & C^*_r\left(J_R \cup \Ro(p)\right) \ar[d] \ar[r]
& 0 \\
0 \ar[r] &  A \ar[r]
\ar[d] &
C^*_r\left(J_R \cup \left(\mathcal I_c \cap \rho(I)\right)\right)   \ar[r] \ar[d] & C^*_r\left(J_R \right) \ar[d] \ar[r]
& 0 \\
  & 0   & 0  & 0   & \\
}
\end{xymatrix}
\end{equation*} 
Here
$$
A = C^*_r\left(\mathcal I_c \cap \rho(I)\right) = \left( \oplus_{i=1}^n 
C(\mathbb T) \otimes \mathbb C^{v_i} \otimes \mathbb K_{c_i}\right)  \oplus \left(\oplus_{i=n+1}^{n+m} \mathbb
C^{v_i}\otimes \mathbb K_{c_i}\right)
$$
where $c_1,c_2, \dots, c_{n+m}$ are critical points in
  $\Omega$ such that $\rho(I) \cap \mathcal I_c = \sqcup_{i=1}^{n+m}
  \Ro(c_i)$, and $c_1, \dots, c_n$ are pre-periodic while $c_{n+1},
  \dots, c_{n+m}$ are not.

\subsubsection{Parabolic stable regions}

Assume now that $x$ is contained in a parabolic stable region $\Omega$. In this case there is no periodic
point in $\Omega$ and we get the following diagram.
\begin{equation}\label{ups7}
\begin{xymatrix}{ 
  & 0 \ar[d]  & 0 \ar[d]  &   & \\
 &  \mathbb K \ar[d] \ar@{=}[r]  &  \mathbb K  \ar[d] &
&  \\
0 \ar[r] &  C^*_r\left(\rho(I) \cap F_R \right) \ar[r]
\ar[d] &
C^*_r\left(R \right)/I  \ar[r] \ar[d]  & C^*_r\left(J_R \right)  \ar[r]\ar@{=}[d] & 0 \\
 0 \ar[r] &  \oplus_{i=1}^{N} \mathbb
C^{v_i}\otimes \mathbb K_{c_i} 
\ar[d] \ar[r] & C^*_r\left(J_R \cup \left( \mathcal I_c\cap
    \rho(I)\right) \right) \ar[r]\ar[d]   &   \J \ar[r]
&  0 \\
  & 0   &  0 &    & \\
}
\end{xymatrix}
\end{equation} 
where $c_i, i = 1,2, \dots, N$, are critical points in $\Omega$ such
that $\mathcal I_c \cap \rho(I) = \bigsqcup_{i=1}^N \Ro(c_i)$.

\subsubsection{Stable regions of Siegel
  or Herman type}

Assume now that $x$ is contained in a stable region $\Omega$ of Siegel type.  In this case there is a periodic
point in $\Omega$ with a non-critical orbit, but since $x \notin
\mathcal I$ this orbit is not in $\rho(I)$. Therefore the picture is
the same as in the case of a Herman type stable region and we get in
both cases a
diagram similar to the parabolic case. The only difference is that the algebra
$\mathbb K$ in (\ref{ups7}) is exchanged with the stabilised irrational rotation algebra
$A_{\theta}$ %Specifically, we
%get in both cases a diagram
%\begin{equation}\label{3x9}
%\begin{xymatrix}{ 
%  & 0 \ar[d]  & 0 \ar[d]  &   & \\
% &  A_{\theta}  \ar[d] \ar@{=}[r]  &  A_{\theta}  \ar[d] &
%&  \\
%0 \ar[r] &  C^*_r\left(\rho(I) \cap F_R\right) \ar[r]
%\ar[d] &
%C^*_r\left(R \right)/I  \ar[r] \ar[d]  & C^*_r\left(J_R \right)  \ar[r]\ar@{=}[d] & 0 \\
% 0 \ar[r] &  \oplus_{i=1}^{N} \mathbb
%C^{v_i}\otimes \mathbb K_{x_i} 
%\ar[d] \ar[r] & C^*_r\left(J_R \cup \left( \mathcal I_c\cap
%    \rho(I)\right) \right)  \ar[r]\ar[d]   &   \J \ar[r]
%&  0 \\
%  & 0   &  0 &    & \\
%}
%\end{xymatrix}
%\end{equation} 
%where $A_{\theta}$ is the irrational rotation algebras 
corresponding
to the rotation by the angle $2 \pi \theta $ in the core of $\Omega$.

% and where $x_i, i = 1,2, \dots, N$,
%represent the $\Ro$-orbits of the critical points in $\rho(I) \cap
%\Omega$, and $v_i$ is the asymptotic valency of $x_i$.

\bigskip

This completes the list of primitive quotients of $C^*_r(R)$. Note that only very few
of the primitive quotients are simple. In fact, the simple quotients
of $C^*_r(R)$ are all matrix algebras $M_n(\mathbb C)$ with $n
\leq 4$, together with $\J$ when
there are no exposed point in $J_R$.

\end{document}